\documentclass[11pt]{article}
\usepackage{tikz}
\usepackage{latexsym}
\usepackage{amsthm}
\usepackage{amsmath}
\usepackage{graphicx}
\usepackage{amssymb}
\usepackage{color}
\usepackage{ulem}
\usepackage{epsfig}

\newcommand{\lng}[1]{\ell({#1})}
\newcommand{\lngs}[1]{\lng{\itvs{#1}}}
\newcommand{\lngt}[2]{\lng{\itvt{#1}{#2}}}
\newcommand{\lngi}[1]{\lngt{#1}{i}}
\newcommand{\lngcs}[1]{\lng{\cuts{#1}}}
\newcommand{\lngct}[2]{\lng{\cutt{#1}{#2}}}
\newcommand{\lngc}[1]{\lngct{#1}{i}}

\newcommand{\ratiow}{x}
\newcommand{\ratio}[1]{\ratiow_{#1}}
\newcommand{\dratiow}{{\rm dr}}
\newcommand{\dratio}[1]{\dratiow(#1)}

\newcommand{\cutw}{g}
\newcommand{\cudw}{\cutw'}
\newcommand{\itvw}{I}

\newcommand{\itvs}[1]{\itvw_{#1}}
\newcommand{\itvt}[2]{\itvw_{#2_{#1}}^{#1}}
\newcommand{\itvi}[1]{\itvt{#1}{i}}
\newcommand{\itvb}[2]{\itvw_{#2}^{#1}}
\newcommand{\cut}[1]{\cutw(#1)}
\newcommand{\cud}[1]{\cudw(#1)}
\newcommand{\cuts}[1]{\cut{\itvs{#1}}}
\newcommand{\cutt}[2]{\cutw^{#1}(\itvt{#1}{#2})}
\newcommand{\cutb}[2]{\cutw^{#1}(\itvb{#1}{#2})}
\newcommand{\cuti}[1]{\cutt{#1}{i}}
\newcommand{\NCF}{\operatorname{NICF}}
\newcommand{\NF}{\NCF_5}
\newcommand{\NFS}{\NCF_5^*}
\newcommand{\CNF}{C_{\NCF}}
\newcommand{\lbrf}{\langle}
\newcommand{\rbrf}{\rangle}
\newcommand{\NFR}{\NCF_4}
\newcommand{\NFRX}{\NCF_{4,\setminus \QQ}}
\newcommand{\NFRY}{\NCF_{4,\QQ}}
\newcommand{\bara}{{\bar a}}

\newtheorem{theorem}{Theorem}

\newtheorem{lemma}[theorem]{Lemma}
\newtheorem{corollary}[theorem]{Corollary}

\newtheorem{proposition}[theorem]{Proposition}
\newtheorem{remark}[theorem]{Remark}

\newtheorem{definition}[theorem]{Definition}
\newcommand{\NICF}{{\rm NICF}}
\newcommand{\RCF}{{\rm RCF}} 

\newcommand{\RR}{\ensuremath{\mathbb{R}}}     
\newcommand{\QQ}{\ensuremath{\mathbb{Q}}}     
\newcommand{\ZZ}{\ensuremath{\mathbb{Z}}}  
\DeclareMathOperator\sign{sign}

\def\Z{{\mathbb Z}}

\def\R{{\mathbb R}}

\def\calC{{\cal C}}
\def\calD{{\cal D}}

\begin{document}
\title{Sums of nearest integer continued fractions with bounded digits:
$\NICF_5 + \NICF_5 = \R$}
\author{Wieb Bosma and Alex Brouwers\\ \\
{\it Department of Mathematics, Radboud Universiteit Nijmegen, NL}\\
\\
email: {\tt bosma@math.ru.nl} and {\tt alex.wm.brouwers@gmail.com}}
\maketitle

\section{Introduction}\label{sec:intro}
Slightly dependent on one's perspective, there are two
main defects that continued fractions suffer as an
alternative (to the decimal, or binary) representation of real numbers: they
require an infinite alphabet for the `digits', and arithmetic is problematic.

In the following theorem of Hall, from 1947, one could say that the first
defect is shifted to the second.

\begin{theorem}[Hall \cite{Hall}, 1947]\label{thm:hall}
For every real number $x$ there exist real numbers $u, v$
such that $x=u+v$ and both $u, v$ have the property that
their regular partial fractions $b_i$ are bounded by 4
(with the possible exception of $b_0$).
\end{theorem}
The partial fractions $b_i=b_i(x)$ are integers in the
regular continued fraction expansion (RCF) of a real number $x$ as
$$\RCF(x)=b_0+\frac{1}{\displaystyle\strut b_1+\frac{1}{\displaystyle\strut b_2+\frac{1}{\displaystyle \ddots}}},$$
and they satisfy
$b_0\in\Z$ and $b_i\geq 1$ (for $i\geq 1$).

The statement of Hall's theorem is often summarized as
$$\RR=\RCF_4+\RCF_4.$$
Much later, in 1973, Cusick \cite{Cusick} and Divis \cite{Divis} showed that
a symmetric improvement on this theorem is not possible, that
is to say, in analogous fashion:
$$\RR\neq\RCF_3+\RCF_3;$$
but Hlavka \cite{Hlavka} showed in 1975 that
$$\RR=\RCF_3+\RCF_4.$$
The current paper is concerned with similar results for
the nearest integer continued fraction expansion (NICF)
of real numbers. In some respects the NICF is `better'
than the RCF; it selects the best approximations and converges
faster, but it is also more natural in the sense that it
allows an obvious generalization to continued fraction
expansions of complex numbers, namely the Hurwitz
expansion. See also \cite{BosmaGruenewald}.

The easiest definition of both regular and nearest integer
continued fractions is by the algorithm used to obtain the
expansions. The RCF expansion
for a real number $x$ can be obtained by taking $x_0=x$
and iterating the steps
$b_i=\lfloor x_i\rfloor$  and $x_{i+1}=\frac{1}{x_i-b_i}$
for $i\geq 0$, as long as $b_i\neq x_i$. It provides essentially
unique expansions for every real number (unique if we insist that
finite expansions, for rational numbers, do not end with $1$),
and every infinite sequence $[b_0; b_1, b_2, \cdots]$ with
$b_i\in\Z$ satisfying $b_i\geq 1$ for $i\geq 1$ occurs as the
RCF expansion of a real number.

The NICF expansion is
obtained by replacing $b_i=\lfloor x_i\rfloor$ by $a_i=\lfloor
x_i\rceil$, which is the integer nearest to $x_i$. An important
difference with RCF is that in the NICF the $x_{i}$, and hence the
$a_i$, may become negative,
so in
$$\NICF(x)=a_0+\frac{1}{\displaystyle\strut a_1+\frac{1}{\displaystyle\strut a_2+\frac{1}{\displaystyle \ddots}}},$$
we have $a_i\in\Z$ for $i\geq 1$;  however,
the following conditions for $i\geq 1$ describe precisely the sequences
of integers that arise as NICF expansions:
\begin{itemize}
\item $a_{i} \in \ZZ \setminus \{-1, 0, 1\}$;
\item if $a_{i} = 2$, then $a_{i+1} \ge 2$;
\item if $a_{i} = -2$, then $a_{i+1} \le -2$.
\end{itemize}
From here on, digits $a_i$ always refer to the nearest integer expansion.

The main result in this paper can be summarized as follows:
\begin{equation}\RR=\NICF_5+\NICF_5.\label{eq:sum}\end{equation}
More formally, the following analogon to Hall's theorem will be proved.
\begin{theorem}\label{maintheorem}
For every real number $x$ there exist real numbers $u, v$
such that $x=u+v$ and both $u, v$ have the property that
their nearest integer partial fractions $a_i$ are bounded by 5
(with the possible exception of $a_0$).
\end{theorem}
We will also prove that
\begin{equation}\RR\neq\NICF_4+\NICF_4.\label{eq:not}\end{equation}
Our result is obtained essentially by adapting the
Cantor set methods of Hall and Hlavka directly
to nearest integer continued fractions. This is explained
in Section \ref{sec:cantor} and Section \ref{sec:compdiv}.

In Section \ref{sec:nicf} we create a Cantor set, $\CNF$, and prove 
in Section \ref{sec:equal} it is equal to $\NFS$,
which is a subset of $\NF$.
With the results of Section \ref{sec:cantor} and Section \ref{sec:compdiv}, we 
show in Section \ref{sec:equal} that $\CNF + \CNF$ contains an interval
of length 1: $\CNF + \CNF\supseteq [\frac{1}{2},\frac{3}{2}]$ from
which the main result, Theorem \ref{maintheorem}, or Equation (\ref{eq:sum}), will follow.
Finally, Section \ref{sec:not4} is devoted to a proof of 
inequality (\ref{eq:not}).

The only previous result in this direction that we
are aware of was mentioned in the M.Sc.~thesis of
Aldenhoven (Nijmegen 2012, \cite{Noud}), namely
$$\RR=\NICF_6+\NICF_6,$$
a result directly obtained from that of Hall by using `singularization'
of regular continued fractions (see \cite{Noud} Section 3.4):
$\RCF_4\subset\NICF_6$.

\section{Cantor sets}\label{sec:cantor}
\noindent
For any closed, real interval $[a, b]$, with $a, b\in\R$ and $a<b$,
we denote its {\sl length} by $\lng{[a,b]}=b-a$; the same
notation is used for the length of open, or half-open real
intervals $(a, b), [a, b), (a, b]$.

The {\sl sum} $A+B$ of two closed real intervals $A=[a_l, a_r], B=[b_l, b_r]$
will be defined by pointwise addition
$$A + B = \{ x + y : x \in A, y \in B\}$$
and similarly for (half-)open intervals. In particular,
for closed intervals
$$A+B=[a_l, a_r] + [b_l, b_r] = [a_l+ b_l, a_r+ b_r].$$
This sum is clearly commutative,
associative, and the length is additive, on closed intervals $A, B, C$:
$$A + B = B + A,\quad (A+B)+C = A+(B+C), \quad\textrm{and}\quad\lng{A+B}= \lng{A} + \lng{B}.$$
The following useful result on the union of intervals is also straightforward.
\begin{lemma} \label{intervalSum}
Let $A$ and $B$ be closed intervals; if 
$$\lng{A} + \lng{B} \ge \lng{[\min(a_l, b_l), 
\max(a_r, b_r)]},$$
then $A \cup B$ is also a closed interval, namely 
$$A \cup B=[\min(a_l, b_l), \max(a_r, b_r)].$$
\end{lemma}
The Cantor sets $\cal C$  we will be constructing arise in the following
way from an initial, closed interval $I_1$ using a {\sl gap function}
$g$: for every interval $I_i$ the function $g$ selects an open
gap interval $g(I_i)\subset I_i$, determining two new closed intervals,
the left interval $I_{2i}$ and the right interval $I_{2i+1}$, where
$I_i=I_{2i}\cup g(I_i)\cup I_{2i+1}$, disjointly.
Pictorially:
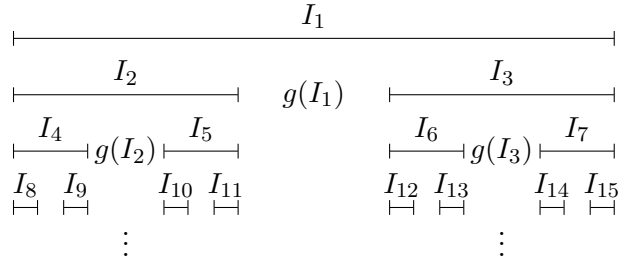
\begin{center}
\begin{tikzpicture}[yscale = 0.75]
   \draw[|-|] (0,5) -- (8,5);
   \node[above] at (4,5) {$\itvs{1}$};
   \draw[|-|] (0,4) -- (3,4);
   \node[above] at (1.5,4) {$\itvs{2}$};
   \draw[|-|] (5,4) -- (8,4);
   \node[above] at (6.5,4) {$\itvs{3}$};
   \node at (4,4) {$\cuts{1}$};
   \draw[|-|] (0,3) -- (1,3);
   \node[above] at (0.5,3) {$\itvs{4}$};
   \draw[|-|] (2,3) -- (3,3);
   \node[above] at (2.5,3) {$\itvs{5}$};
   \node at (1.5,3) {$\cuts{2}$};
   \draw[|-|] (5,3) -- (6,3);
   \node[above] at (5.5,3) {$\itvs{6}$};
   \draw[|-|] (7,3) -- (8,3);
   \node[above] at (7.5,3) {$\itvs{7}$};
   \node at (6.5,3) {$\cuts{3}$};
   \draw[|-|] (0,2) -- (0.3333,2);
   \node[above] at (0.1667,2) {$\itvs{8}$};
   \draw[|-|] (0.6667,2) -- (1,2);
   \node[above] at (0.8333,2) {$\itvs{9}$};
   \draw[|-|] (2,2) -- (2.3333,2);
   \node[above] at (2.1667,2) {$\itvs{10}$};
   \draw[|-|] (2.6667,2) -- (3,2);
   \node[above] at (2.8333,2) {$\itvs{11}$};
   \draw[|-|] (5,2) -- (5.3333,2);
   \node[above] at (5.1667,2) {$\itvs{12}$};
   \draw[|-|] (5.6667,2) -- (6,2);
   \node[above] at (5.8333,2) {$\itvs{13}$};
   \draw[|-|] (7,2) -- (7.3333,2);
   \node[above] at (7.1667,2) {$\itvs{14}$};
   \draw[|-|] (7.6667,2) -- (8,2);
   \node[above] at (7.8333,2) {$\itvs{15}$};
   \node at (1.5,1.5) {$\vdots$};
   \node at (6.5,1.5) {$\vdots$};
   \end{tikzpicture}
\end{center}
The prototype is the original Cantor set 
where $g$ cuts out the middle third of any interval.
The {\sl Cantor set} determined by $I_1$ and $g$ will by definition be the set
of points of $I_1$ not contained in any gap:
$${\cal C}=\calC(I_1,\cutw) = \itvs{1} \setminus \bigcup_{i=1}^\infty \cuts{i}.$$
Note that all the $\cuts{i}$ are disjoint, so $\sum_{i=1}^\infty \lngcs{i} 
\le \lngs{1}$. Because $\lngs{1}$ is finite, we have 
$$\lim_{i \to \infty} \lngcs{i} = 0.$$
Also note that we identify $\calC$ with its subset of points on
the real line, and thus it may be that $\calC=(I_1, f)=(J_1, g)=\calD$ although
their gap functions $f$ and $g$ are different; the initial intervals
will satisfy $I_1=J_1$ in this case.

The relative size of gaps is given by the density ratio.
\begin{definition}
The {\sl density ratio} of Cantor set ${\cal C}=\calC(I_1,\cutw)$
is defined as:
$$\dratiow_{\cal C}=\dratio{I_1 ,\cutw} = \inf_{i=1}^\infty \frac{\min(\lng{I_{2i}},
	\lng{I_{2i+1}})}{\lng{\cut{I_i}}}.$$
\end{definition}

\smallskip\noindent
It follows that $\lngs{2i} \ge \dratiow_{\cal C}\cdot \lngcs{i}$ 
and $\lngs{2i+1} \ge \dratiow_{\cal C} \cdot \lngcs{i}$,  for all $i$.

\begin{definition}\rm
We call a Cantor set {\it hole-decreasing} if for all $i$ we have both
$\lngcs{2i} \le \lngcs{i}$ and $\lngcs{2i+1} \le \lngcs{i}$.
\end{definition}

\smallskip\noindent
Equivalently:  a Cantor set is hole-decreasing if and only if 
for every $i$, for every $k$ such that $I_k \subseteq I_i$ we have
$\lngcs{k} \le \lngcs{i}$.

If we call a Cantor set {\it $n$-hole-decreasing} if for every $i \le n$,
for every $k$ such that $I_k \subseteq I_i$ we have
$\lngcs{k} \le \lngcs{i}$, it will be clear that $\calC$ is 
hole-decreasing if and only if 
$\calC$ is $n$-hole-decreasing for every $n$.

By repeatedly interchanging an interval with its remainders, 
depicted below, it is possible to transform the construction
of any Cantor set into a hole-decreasing construction. Suppose,
for example, that $\lngcs{2i}>\lngcs{i}$:
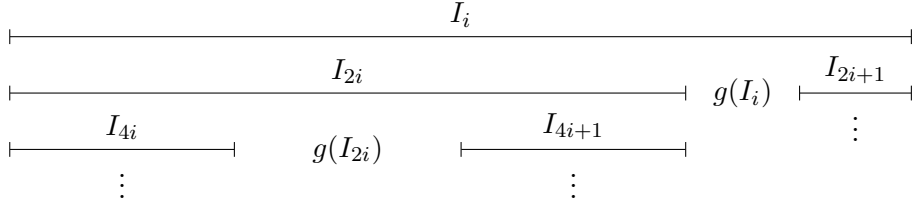
\begin{center}
\begin{tikzpicture}[yscale = 0.75,xscale = 1.5]
\draw[|-|] (0,5) -- (8,5);
\node[above] at (4,5) {$\itvs{i}$};
\draw[|-|] (0,4) -- (6,4);
\node[above] at (3,4) {$\itvs{2i}$};
\draw[|-|] (7,4) -- (8,4);
\node[above] at (7.5,4) {$\itvs{2i+1}$};
\node at (6.5,4) {$\cuts{i}$};
\draw[|-|] (0,3) -- (2,3);
\node[above] at (1,3) {$\itvs{4i}$};
\draw[|-|] (4,3) -- (6,3);
\node[above] at (5,3) {$\itvs{4i+1}$};
\node at (3,3) {$\cuts{2i}$};
\node at (7.5,3.5) {$\vdots$};
\node at (1,2.5) {$\vdots$};
\node at (5,2.5) {$\vdots$};
\end{tikzpicture}
\end{center}
then combine the subintervals on the right to
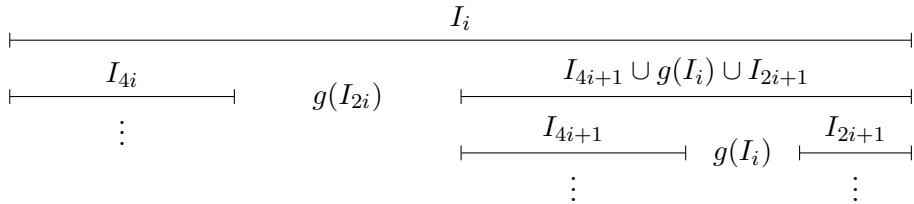
\begin{center}
\begin{tikzpicture}[yscale = 0.75,xscale = 1.5]
\draw[|-|] (0,5) -- (8,5);
\node[above] at (4,5) {$\itvs{i}$};
\draw[|-|] (0,4) -- (2,4);
\node[above] at (1,4) {$\itvs{4i}$};
\draw[|-|] (4,4) -- (8,4);
\node[above] at (6,4) {$\itvs{4i+1}\cup\cuts{i}\cup\itvs{2i+1}$};
\node at (3,4) {$\cuts{2i}$};
\draw[|-|] (4,3) -- (6,3);
\node[above] at (5,3) {$\itvs{4i+1}$};
\draw[|-|] (7,3) -- (8,3);
\node[above] at (7.5,3) {$\itvs{2i+1}$};
\node at (6.5,3) {$\cuts{i}$};
\node at (1,3.5) {$\vdots$};
\node at (5,2.5) {$\vdots$};
\node at (7.5,2.5) {$\vdots$};
\end{tikzpicture}
\end{center}
and rename, to arrive at
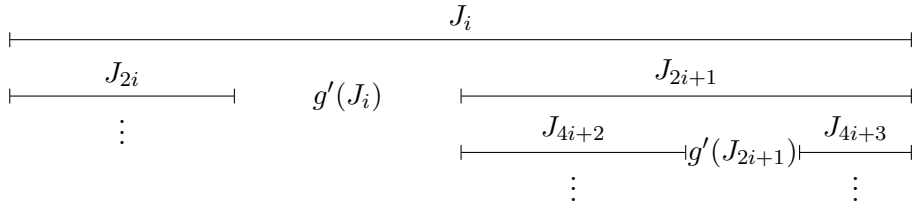
\begin{center}
\begin{tikzpicture}[yscale = 0.75, xscale = 1.5]
\draw[|-|] (0,5) -- (8,5);
\node[above] at (4,5) {$J_i$};
\draw[|-|] (0,4) -- (2,4);
\node[above] at (1,4) {$J_{2i}$};
\draw[|-|] (4,4) -- (8,4);
\node[above] at (6,4) {$J_{2i+1}$};
\node at (3,4) {$\cud{J_i}$};
\draw[|-|] (4,3) -- (6,3);
\node[above] at (5,3) {$J_{4i+2}$};
\draw[|-|] (7,3) -- (8,3);
\node[above] at (7.5,3) {$J_{4i+3}$};
\node at (6.5,3) {$\cud{J_{2i+1}}$};
\node at (1,3.5) {$\vdots$};
\node at (5,2.5) {$\vdots$};
\node at (7.5,2.5) {$\vdots$};
\end{tikzpicture}
\end{center}
clearly satisfying
$$
\frac{\lng{J_{2i}}}{\lng{\cud{J_i}}} = \frac{\lngs{4i}}{\lngcs{2i}}\
\quad\textrm{and}\quad
\frac{\lng{J_{2i+1}}}{\lng{\cud{J_i}}} >
\frac{\lngs{4i+1}}{\lng{\cud{J_{i}}}},
$$
while
$$
\frac{\lng{J_{4i+2}}}{\lng{\cud{J_{2i+1}}}} =
\frac{\lngs{4i+1}}{\lngcs{i}} > \frac{\lngs{4i+1}}{\lngcs{2i}}
\quad\textrm{and}\quad
\frac{\lng{J_{4i+3}}}{\lng{\cud{J_{2i+1}}}} = \frac{\lngs{2i+1}}{\lngcs{i}},
$$
from which it follows that 
$\dratio{J_1, g'}\ge\dratio{I_1,g}$.
We obtain the following.

\begin{proposition} \label{holedecrease}
For every Cantor set $\cal C$ there exists a hole-decreasing Cantor set
$\cal D$ such that ${\cal D} = {\cal C}$ and $\dratiow_{\cal D}\ge\dratiow_{\cal C}$.
\end{proposition}

\section{Comparable and Dividable}\label{sec:compdiv}
In Section \ref{sec:nicf} we will see that bounded continued fractions
determine Cantor sets; for their sum we will want to add Cantor
sets, and this is defined in the obvious, pointwise, fashion:
${\cal C}_1+{\cal C}_2=\{ x+y: x\in {\cal C}_1, y\in{\cal C}_2\}$.
In this section, we give density requirements sufficient
to ensure for sums of Cantor
sets to `fill each other's gaps', that is, such that
${\cal C}_1+{\cal C}_2=I_1+I_2$, where ${\cal C}_1$ and ${\cal C}_2$
are constructed from the initial intervals $I_1$ and $I_2$.
This is done in a similar way as done by Hlavka
in \cite{Hlavka}, with {\it comparable} and {\it dividable} intervals.
At the end of this section we show that our requirements are slightly
weaker than Hlavka's.

First a lemma.
\begin{lemma}
For a subinterval $I_i$ of a Cantor set ${\cal C}=(I, g)$
and any interval~$J$:
\begin{equation} \label{sumofintervals}
\lng{J} \ge \lng{\cut{I_i}}\quad\Rightarrow\quad
 (I_{2i} + J) \cup (I_{2i+1} + J) = I_i + J.
\end{equation}
\end{lemma}
\begin{proof}
Apply Lemma \ref{intervalSum} with $A = I_{2i} + J$ and $B = I_{2i+1} + J$
together with:
$$\lng{I_i + J} = \lng{I_{2i}} + \lng{\cut{I_i}} + \lng{I_{2i+1}} + \lng{J} \le  \lng{I_{2i} + J} + \lng{I_{2i+1} + J}.$$
\end{proof}

\begin{theorem} \label{compareanddivideTheorem}
Suppose that ${\cal C}_1 = (I_1^1,\cutw^1), \ldots, {\cal C}_n =
(I_1^n,\cutw^n)$ are $n\geq 2$ 
Cantor sets for which the density ratios satisfy  
$\dratiow_{{\cal C}_i}\ge \ratio{i}$, for $1\leq i\leq n$, with
\begin{equation}\label{largerthan1} 
  \sum_{i=1}^n \frac{\ratio{i}}{\ratio{i} + 1} \ge 1
\end{equation}
and
\begin{equation}\label{initialCompare}
 \forall j,k:\ \lng{\itvb{j}{1}} \ge 
  \frac{\ratio{j}}{\ratio{j} + 1}(\ratio{k} + 1) \lng{\itvb{k}{1}}. 
\end{equation}
Then 
$${\cal C}_1+{\cal C}_2+\cdots +{\cal C}_n=I_1^1+I_1^2+\cdots+I_1^n.$$
\end{theorem}

\smallskip\noindent
The remainder of this section is geared towards a proof of this
theorem. 

First we note that we may as well assume that the Cantor sets
in the theorem are hole-decreasing, since by Proposition \ref{holedecrease}
there exist Cantor sets ${\cal D}_i$ with $\calD_i=\calC_i$
that are hole-decreasing and for which
$\dratiow_{{\cal D}_i}\geq\dratiow_{{\cal C}_i}$.

\begin{definition}\rm
Let $\itvi{j}$ for $1\leq j\leq n$, be subintervals
in the construction of hole-decreasing Cantor sets
${\cal C}_j=(I_1^j, g^j)$ for which
$\dratio{{\cal C}_j} \ge \ratio{j}$
such that (\ref{largerthan1}) and (\ref{initialCompare}) hold.
Then $\itvi{1}, \itvi{2}, \ldots, \itvi{n}$ are called {\it comparable} if 
$$\forall j,k:\ \lngi{j} \ge \frac{\ratio{j}}{\ratio{j} + 1} (\ratio{k} + 1) \lngc{k}.$$
\end{definition}

\smallskip\noindent
Note that the initial intervals 
$\itvb{1}{1},\itvb{2}{1}, \ldots, \itvb{n}{1}$ for such Cantor sets
are always comparable since 
$\lng{\itvb{k}{1}} \ge \lng{g^k(\itvb{k}{1})}$.
 \begin{lemma}
If $\itvi{1}, \itvi{2}, \ldots, \itvi{n}$ are comparable intervals, 
then for all $k$:
$$\sum_{j=1}^n \itvi{j} = 
 \left( \sum_{j=1;j \ne k}^n \itvi{j} + \itvt{k}{2i} \right) \cup 
 \left(\sum_{j=1;j \ne k}^n \itvi{j} + \itvb{k}{2i_k+1}\right).$$
\end{lemma}
\begin{proof}
We would like to apply (\ref{sumofintervals}) with $\sum_{j=1;j \ne k}^n \itvi{j}$ 
and $\itvi{k}$ for $J$ and $I_i$, respectively.
What remains to be shown is that $\lng{\sum_{j=1;j \ne k}^n \itvi{j}} \ge 
\lngc{k}$. Indeed, comparability implies
\begin{align*}
\sum_{j=1;j \ne k}^n \lngi{j} &\ge \left(\sum_{j=1;j \ne k}^n 
	\frac{\ratio{j}}{\ratio{j} + 1}\right) (\ratio{k} + 1) \lngc{k}\\
&=  \left(\sum_{j=1}^n \frac{\ratio{j}}{\ratio{j} + 1} - 
	\frac{\ratio{k}}{\ratio{k} + 1}\right) (\ratio{k} + 1) \lngc{k}\\
&\overset{\eqref{largerthan1}}{\ge} 
    \left(1 - \frac{\ratio{k}}{\ratio{k} + 1}\right) (\ratio{k} + 1) \lngc{k}\\
&= \lngc{k}.
\end{align*}
\end{proof}

\begin{definition}\rm
Comparable intervals $\itvi{1}, \itvi{2}, \ldots, \itvi{n}$ are
$j$-{\it dividable} if both 
$\itvi{1}, \itvi{2}, \ldots, \itvb{j}{2i_j},\ldots,\itvi{n}$ and 
$\itvi{1}, \itvi{2}, \ldots, \itvb{j}{2i_j+1},\ldots,\itvi{n}$ 
are comparable intervals.
\end{definition}

\begin{lemma}\label{comparableimplies}
Let comparable intervals $\itvi{1}, \itvi{2}, \ldots, \itvi{n}$ be given. 
If $j$ is such that $(\ratio{j} + 1) \lngc{j}$ is maximal, i.e.,
\[\forall k : (\ratio{j} + 1) \lngc{j} \ge (\ratio{k} + 1) \lngc{k},\]
then $\itvi{1}, \itvi{2}, \ldots, \itvi{n}$ are $j$-dividable.
\end{lemma}
\begin{proof}
Since the Cantor sets are hole-decreasing, we have that 
$\lng{\cutb{j}{2i_j}} \le \lngc{j}$ and $\lng{\cutb{j}{2i_j+1}} \le \lngc{j}$.
We still have to prove
$$\forall k: \lng{\itvb{j}{2j}} \ge 
	\frac{\ratio{j}}{\ratio{j} + 1} (\ratio{k} + 1) \lngc{k}$$
and
$$\forall k: \lng{\itvb{j}{2j+1}} \ge 
	\frac{\ratio{j}}{\ratio{j} + 1} (\ratio{k} + 1) \lngc{k}.$$ 
As both 
$\lng{\itvb{j}{2i_j}} \ge \ratio{j}\lngc{j}$ and 
$\lng{\itvb{j}{2i_j+1}} \ge \ratio{j}\lngc{j}$, 
it is sufficient to note that: 
$$\ratio{j}\lngc{j} \ge \ratio{j}\lngc{j} 
\frac{(\ratio{k}+1)\lngc{k}}{(\ratio{j}+1)\lngc{j}} = 
\frac{\ratio{j}}{\ratio{j} + 1} (\ratio{k} + 1) \lngc{k}.$$
\end{proof}

\begin{proof}[Proof of Theorem \ref{compareanddivideTheorem}]
We will recursively create sets $G_m$ on $n$-tuples of intervals.
Start with
$$G_0 = \{\langle \itvb{1}{1},\itvb{2}{1}, \ldots, \itvb{n}{1} \rangle\};$$
and then, for $m\geq 0$:
\begin{eqnarray*}
	G_{m+1}&=&\left\{\langle \itvi{1},\itvi{2}, \ldots,
	\itvb{k}{2i_{k}}, \ldots, \itvi{n} \rangle,\langle
	\itvi{1}, \itvi{2}, \ldots, \itvb{k}{2i_{k}+1}, \ldots,
	\itvi{n} \rangle\right.\ \ \vert\\
&&\quad \langle \itvi{1}, \itvi{2}, \ldots, \itvi{k}, \ldots, \itvi{n}\rangle \in G_m\quad\land\\
&&\quad\quad\left. \forall l : x_k \cdot \lng{\cutb{k}{i_k}}\ge x_l\cdot \lng{\cutb{l}{i_l}} \right\}.
\end{eqnarray*}
Trivially, every $\langle \itvi{1}, \itvi{2}, \ldots, 
\itvi{n} \rangle \in G_0$ is comparable. With induction
and Lemma \ref{comparableimplies}, we have that
$\langle \itvi{1}, \itvi{2}, \ldots, \itvi{n} \rangle \in G_m$
is comparable for every $m$, and therefore

$$\bigcup \left\{ \sum \itvi{j} \middle| \langle \itvi{1}, \itvi{2}, 
\ldots, \itvi{n} \rangle \in G_m \right\} = \sum \itvb{j}{1}.$$

Let $\cuti{k}$ be a gap of $C^k$, then for every $j$ there 
are only a finite number $a_j$ of gaps $\cutb{j}{l}$ such that 
$\lngct{j}{l} \ge \frac{\ratio{k}}{\ratio{j}}\lngct{k}{i}$. 
For every $a \ge \sum a_j$, there will be no  $\langle \itvi{1},
\itvi{2}, \ldots, \itvi{n} \rangle \in G_a$ such that 
$\cuti{k} \subseteq\itvi{k}$.

Let us define $G$ as:
$$G = \lim_{i \to \infty} G_i.$$
Then for each $\langle \itvi{1}, \itvi{2}, \ldots, \itvi{n} \rangle \in G$ 
we have $\itvi{j} \subset C^j$. 
Therefore $$\left\{\sum_{i=1}^n x_i \middle| x_i \in C^i\right\}
\subseteq \bigcup \left\{ \sum c_j \middle| \langle \itvi{1},
\itvi{2}, \ldots, \itvi{n} \rangle \in G \land c_j \in \itvi{j}
 \right\} = \sum \itvb{j}{1}.$$
\end{proof}
\noindent
We give one more theorem, which is less general but will be easier to use.
\begin{theorem} \label{usefulCantor}
Given Cantor sets ${\cal C}^1, {\cal C}^2, \ldots {\cal C}^n$ with 
density ratios equal to $\ratio{1}, \ratio{2}, \ldots, \ratio{n}$ such that
\begin{equation}
    \sum_{i=1}^n \frac{\ratio{i}}{\ratio{i} + 1} \ge 1;
\end{equation}
and
\begin{equation}
    \forall j,k: \lng{\itvb{j}{1}} \ge \frac{\ratio{j}}{\ratio{j} + 1} 
	    \frac{\ratio{k} + 1}{2\ratio{k} + 1} \lng{\itvb{k}{1}}.
\end{equation}
Then $$\sum_{i=1}^n I_1^i = \{\sum_{i=1}^n x_i : x_i \in {\cal C}^i\}.$$
\end{theorem}
\begin{proof}
  First, notice that:
  \begin{align*}
      \lng{\itvb{k}{i_k}} &= \lng{\itvb{k}{2i_k}} + \lng{\cut{\itvb{k}{i_k}}} + 
\lng{\itvb{k}{2i_k+1}}\\[5pt]
      &= \frac{\lng{\itvb{k}{2i_k}} + \lng{\cut{\itvb{k}{i_k}}} + \lng{\itvb{k}{2i_k+1}}}{\lng{\cut{\itvb{k}{i_k}}}} \cdot \lng{\cut{\itvb{k}{i_k}}}\\[5pt]
      &\ge (2\ratio{k} + 1)\lng{\cut{\itvb{k}{i}}}.
  \end{align*}

  Since $\lng{\itvb{k}{1}} \ge \lng{\itvb{k}{j}}$ for all $j$, $\max_{i} \lng{\cut{\itvb{k}{i}}} \le \frac{1}{2\ratio{k} + 1}\lng{\itvb{k}{1}}$.

  So, for all $j,k$:

  $$\lng{\itvb{j}{1}} \ge \frac{\ratio{j}}{\ratio{j} + 1} \frac{\ratio{k} + 1}{2
\ratio{k} + 1} \lng{\itvb{k}{1}} \ge \frac{\ratio{j}}{\ratio{j} + 1} (\ratio{k} 
+ 1)\max_{i} \lng{\cut{\itvb{k}{i}}}.$$

  We finish by applying Theorem \ref{compareanddivideTheorem}.
\end{proof}

\begin{remark}\rm
In \cite{Hlavka}, Hlavka has a similar result on sums of Cantor
sets; he defines, for
a Cantor set ${\cal C}=(I, g)$ the {\it relative smallest remainder}
	$H_{\cal C}$ and {\it the relative biggest gap} $G_{\cal C}$ by
$$H_{\cal C}=\min\left(\min_i\frac{\lngs{2i}}{\lngs{i}},\,
	\min_i\frac{\lngs{2i+1}}{\lngs{i}}\right)\quad\textrm{and}\quad
G_{\cal C}=\max_i \frac{\lngcs{i}}{\lngs{i}}.$$
Clearly, $\dratiow_{\cal C}\geq \frac{H_{\cal C}}{G_{\cal C}}$.
He then proves, see \cite[Theorem 3]{Hlavka}, that ${\cal C}_1+{\cal C}_2=I^1+I^2$
provided that the following are satisfied:
\begin{equation} \label{hlavka1.1}
 G_{{\cal C}_1} \cdot G_{{\cal C}_2} \le H_{{\cal C}_1} \cdot H_{{\cal C}_2},
\end{equation}
\begin{equation} \label{hlavka1.2}
 G_{{\cal C}_1} \cdot \lng{I^1} \le \lng{I^2},
\end{equation}
and
\begin{equation} \label{hlavka1.3}
 G_{{\cal C}_2} \cdot \lng{I^2} \le \lng{I^1}.
\end{equation}
Now it is not difficult to show (cf.\ \cite{thesis} for details)
that (\ref{hlavka1.1}), (\ref{hlavka1.2}),
and (\ref{hlavka1.3}) imply the existence of $x_1\leq\dratiow_{{\cal C}_1}$
and $x_2\leq\dratiow_{{\cal C}_2}$ such that (\ref{largerthan1}) and
(\ref{initialCompare}) are satisfied. 
Thus \cite[Theorem 3]{Hlavka} follows from Theorem 
\ref{compareanddivideTheorem}.
\end{remark}
\section{Construction of the Cantor set}\label{sec:nicf}
We now turn to the nearest integer continued fraction expansions, and first
introduce the subsets of reals with bounded expansion.
\begin{definition}\label{def:nicfr}
$\NCF_r$ is a subset of $\RR$ given by
\[\NCF_r = \left\{x : x \in \RR \mid \NCF(x) = [a_0;a_1, a_2,\ldots] \text{ and }\ |a_i| \le r \text{ for } {i \ge 1}\right\}. \]
\end{definition}
\noindent
From the rules for $\NCF$ given in the Introduction, it follows that
$x \in \NF$, with $\NICF(x)= [a_0;a_1, a_2,\ldots]$,
if and only if for all $i \ge 1$:
\begin{itemize}
\item $a_{i} \in \{-5, -4, -3, -2, 2, 3, 4, 5\}$; 
\item if $a_{i} = 2$, then $a_{i+1} \in \{2, 3, 4, 5\}$; 
\item if $a_{i} = -2$, then $a_{i+1} \in \{-2,-3,-4,-5\}$.
\end{itemize}

\begin{definition}\label{def:nicfstar}\rm
$\NFS$ is the subset of $\NF \cap [0,1]$ containing only the numbers 
represented by a nearest integer continued fraction satisfying:
$$\text{if }\vert a_i\vert = 5\text{ then }\sign{a_i}\neq\sign{a_{i+1}}.$$
\end{definition}
\noindent
As a consequence, $\NFS$ satisfies additional rules
to the ones for $\NF$:
if $\NICF(x)=[a_0; a_1, a_2, \ldots]$ then $x\in\NFS$ if and only if
\begin{itemize}
\item $a_0 \in \{0,1\}$;
\item if $a_0 = 0$, then  $a_{1} \in \{ 2, 3, 4, 5\}$; 
\item if $a_0 = 1$, then  $a_{1} \in \{ -2, -3, -4, -5\}$;
\item if $a_{i} = 5$, then $a_{i+1} \in \{-2, -3, -4, -5\}$; 
\item if $a_{i} = -5$, then $a_{i+1} \in \{ 2, 3, 4, 5\}$.
\end{itemize}
It easily follows that for all $z \in \ZZ$ and $x \in \NFS$ 
we have $z + x \in  \NF$.

We will show that $\NFS \setminus \QQ$ can be described as a Cantor set,
with density ratio greater than 1, and an initial interval with size
greater than $\frac{1}{2}$.

Let $\mu\in\RR$ be defined by $\mu=\frac{5 + \sqrt{21}}{2}=4.791287\cdots$.
\begin{lemma}
Both $\frac{1}{\mu}$ and $1-\frac{1}{\mu}$ are in 
$\NFS \setminus \QQ$, and for all $x$ in $\NFS$ holds
$$\frac{5-\sqrt{21}}{2}=\frac{1}{\mu}
\leq x\leq 1-\frac{1}{\mu}=\frac{\sqrt{21}-3}{2};$$
also
$$\left[\frac{1}{\mu}, 1-\frac{1}{\mu}\right]\supset [0.20872, 0.79128].$$
\end{lemma}
\begin{proof}
The main observations for the proof are that
$$\NICF(\mu)=[5;\overline{-5, 5}],\quad
\NICF(\frac{1}{\mu})=[0,\overline{5, -5}],\quad
\NICF(1-\frac{1}{\mu})=[1,\overline{-5, 5}];$$
where the bar indicates the period, which is infinitely repeated,
and $\mu=5-\frac{1}{\mu}$. The inequalities follow,
using the standard comparison result for infinite
continued fractions:
$[a_0; a_1, a_2, \ldots, a_n, \ldots]<[a_0; a_1, a_2, \ldots, b_n, \ldots]$
if and only if either $n$ is even and $a_n<b_n$ or $n$ is odd and $a_n>b_n$.
Finally, $\frac{1}{\mu}=0.208712\cdots$ and $1-\frac{1}{\mu}=\mu-4=0.791287\cdots$.
\end{proof}

To create a Cantor set for \NICF, we define a collection 
of intervals that will be used for the gap function.

\begin{definition}\label{intervaldef}
For $a_0, a_1, \ldots, a_n \in \ZZ$ and $k \in \{-5,-4,-3,2,3,4\}$ 
such that $[a_0; a_1, \ldots, a_n, k] \in \NFS$, define 
$$P_{k^+}([a_0, a_1, \ldots, a_n]) := [a_0; a_1, \ldots, a_n, k: \mu] \in \NFS$$
and similarly, for $a_0, a_1, \ldots, a_n \in \ZZ$ and $k \in \{-4,-3,-2,3,4,5\}$ such that
$[a_0; a_1, \ldots, a_n, k] \in \NFS$, define 
$$P_{k^-}([a_0, a_1, \ldots, a_n]) := [a_0; a_1, \ldots, a_n, k: -\mu] \in \NFS.$$
Let $u \in \{-5,-4,-3,2,3,4\}$ and $v \in \{-4,-3,-2,3,4,5\}$, 
with $u<v$, and $a_0,a_1, \dots, a_n \in \ZZ$ such that
$[a_0; a_1, \ldots, a_n, u]$, $[a_0; a_1, \ldots, a_n, v] \in \NFS$;
then 
$T_{u,v}([a_0, a_1, \ldots a_n])$ will be the closed interval with endpoints 
$$P_{u^+}([a_0, a_1, \ldots, a_n])\quad\textrm{and}\quad P_{v^-}([a_0, a_1, \ldots, a_n]).$$
\end{definition}
\noindent
Note that if $n$ is odd then $P_{u^+}([a_0, a_1, \ldots, a_n]) < P_{v^-}([a_0, a_1, \ldots, a_n])$, while if $n$ is even, then $P_{u^+}([a_0, a_1, \ldots, a_n]) > P_{v^-}([a_0, a_1, \ldots, a_n])$.

With the intervals defined in Definition \ref{intervaldef}, we create a
Cantor set. The initial interval will be 
$\left[\frac{1}{\mu}, 1-\frac{1}{\mu}\right]=\left[ 
[0;\overline{5,-5}],[1;\overline{-5,5}]\right]$, 
which we call $T_{0,1}$.
Our gap function $\cutw$ will create remaining intervals of the types:
$$T_{0,1}, T_{2,5}, T_{3,5}, T_{-5,-2}, T_{-5,-3}, \textrm{\ and\ }
T_{b,b+1} \textrm{\ for\ } b \in \{-5,-4,-3,2,3,4\}.$$
For each of these types of interval we describe the function
$\cutw$, and will show that  the remainders again are of the
given types. 

We will also calculate (a lower bound of) the
ratio between the remainders and the size of the gap. Later on,
we will use this to derive a lower bound for the density ratio
of the Cantor set we are creating.
For this the following general considerations will be applied
to intervals $T$ of the above type, and to the gaps we will create
for those.

Let $[a_0;a_1, \ldots, a_n: I^-]$ and $[a_0;a_1, \ldots, a_n: I^+]$ 
be endpoints of an interval $T$, and let $[a_0;a_1, \ldots, a_n:C^-]$ 
and $[a_0;a_1, \ldots, a_n: C^+]$ be the endpoints of a gap, 
where $I^- < C^- < C^+ < I^+$. Here we assume that $[a_0;a_1, \ldots, a_i]$
is the \NICF-expansion of the rational numbers $p_i/q_i$, for
$1\leq i\leq n$.
It is then a standard result on continued fractions that
with $\omega = \frac{q_{n-1}}{q_n}$ 
$$\frac{\lvert I^- - C^- \rvert}{q_n^2 (I^- + \omega) (C^- + \omega)} \qquad \text{and} \qquad \frac{\lvert I^+ - C^+ \rvert}{q_n^2 (I^+ + \omega) (C^+ + \omega)}$$
are the sizes of the remainder intervals,
while the gap has size 
$$\frac{\lvert C^- - C^+ \rvert}{q_n^2 (C^- + \omega) (C^+ + \omega)}.$$ 
Here,
$\lvert\omega\rvert \le \frac{\sqrt{5} - 1}{2}$,
since all endpoints are elements of $\NCF$,  cf.~\cite[p.~378]{Hurwitz}.
The density ratio of this particular interval is then the minimum of 
$$\frac{\lvert I^- - C^- \rvert}{\lvert C^- - C^+ \rvert} \frac{(C^+ + \omega)}{(I^- + \omega)} \qquad \text{and} \qquad \frac{\lvert I^+ - C^+ \rvert}{\lvert C^- - C^+ \rvert} \frac{(C^- + \omega)}{(I^+ + \omega)}.$$

The first type to be considered is that of the single interval 
$T_{0,1}$, which can not be described in the
fashion of Definition \ref{intervaldef}. It has size 
\begin{equation}[1;\overline{-5,5}] - [0;\overline{5,-5}]=\frac{\sqrt{21} - 3}{2} - \frac{5 - \sqrt{21}}{2} = \sqrt{21} - 4 > 0.58256.\label{eq:initsize}\end{equation}
The corresponding gap will be $\cut{T_{0,1}} = ([0;2:\mu] , [1;-2:-\mu])$, which has size: $0.54725\cdots - 0.45275\cdots < 0.09451$.
The remaining intervals are $T_{2,5}([0])$ and $T_{-5,-2}([1])$, which both have size $> 0.24403$.
The density ratio corresponding with our initial interval exceeds $\frac{0.24403}{0.09451} > 2.58205$.
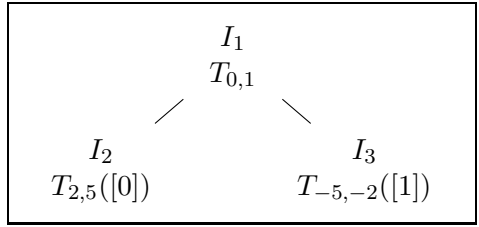
\begin{center}
\fbox{
\begin{tikzpicture}[style={sibling distance = 3.5cm, level distance = 1.5cm}]
  \node {\begin{tabular}{c}$I_1$\\ $T_{0,1}$\end{tabular}}
    child { node {\begin{tabular}{c}$I_2$\\ $T_{2,5}([0])$\end{tabular}}} 
    child { node {\begin{tabular}{c}$I_3$\\ $T_{-5,-2}([1])$\end{tabular}}};
\end{tikzpicture}
}
\end{center}

For the other types of interval we summarize the corresponding information
about the gap in a table and a diagram.

\begin{table}[ht]\tiny
\begin{center}
\begin{tabular}{l|ccccc}
\hline
\textrm{Interval}& $C^-$ & $C^+$ &\multicolumn{2}{c}{\textrm remaining intervals}&\textrm{density ratio}\\
\hline
$T_{0,1}$ & $[0;2:\mu]$ & $[1;-2:-\mu]$ & $T_{2,5}([0])$ & $T_{-5,-2}([1])$ & $> 2.58205 $ \\
$T_{2,3}[\bara]$ & $P_{2^+}([\bara, 2])$ & $P_{-2^-}([\bara, 3])$ & $T_{2,5}([\bara, 2])$ & $T_{-5,-2} ([\bara, 3])$ & $> 2.89191, 2.18031$ \\
$T_{3,4}[\bara]$ & $P_{2^+}([\bara, 3])$ & $P_{-2^-}([\bara, 4])$ & $T_{2,5}([\bara, 3])$ & $T_{-5,-2} ([\bara, 4])$ & $> 2.81108, 2.30709$ \\
$T_{4,5}[\bara]$ & $P_{2^+}([\bara, 4])$ & $P_{-2^-}([\bara, 5])$ & $T_{2,5}([\bara, 4])$ & $T_{-5,-2} ([\bara, 5])$ & $> 2.76375, 2.37311$ \\
$T_{-3,-2}[\bara]$ & $P_{2^+}([\bara, -3])$ & $P_{-2^-}([\bara, -2])$ & $T_{2,5}([\bara, -3])$ & $T_{-5,-2} ([\bara, -2])$ & $> 2.18031, 2.89191$ \\
$T_{-4,-3}[\bara]$ & $P_{2^+}([\bara, -4])$ & $P_{-2^-}([\bara, -3])$ & $T_{2,5}([\bara, -4])$ & $T_{-5,-2} ([\bara, -3])$ & $> 2.30709, 2.81108$ \\
$T_{-5,-4}[\bara]$ & $P_{2^+}([\bara, -5])$ & $P_{-2^-}([\bara, -4])$ & $T_{2,5}([\bara, -5])$ & $T_{-5,-2} ([\bara, -4])$ & $> 2.37311, 2.76375$ \\
$T_{2,5}[\bara]$ & $P_{3^-}([\bara])$ & $P_{3^+}([\bara])$ & $T_{2,3}([\bara])$ & $T_{3,5} ([\bara])$ & $>1.88937, 1.97434$ \\
$T_{3,5}[\bara]$ & $P_{4^-}([\bara])$ & $P_{4^+}([\bara])$ & $T_{3,4}([\bara])$ & $T_{4,5} ([\bara])$ & $>1.76035, 1.06122$ \\
$T_{-5,-2}[\bara]$ & $P_{-3^-}([\bara])$ & $P_{-3^+}([\bara])$ & $T_{-5,-3}([\bara])$ & $T_{-3,-2} ([\bara])$ & $> 1.97434, 1.88937$ \\
$T_{-5,-3}[\bara]$ & $P_{-4^-}([\bara])$ & $P_{-4^+}([\bara])$ & $T_{-5,-4}([\bara])$ & $T_{-4,-3} ([\bara])$ & $> 1.06122, 1.76035$\\
\hline
\end{tabular}
\caption{Gap, remaining intervals, and density ratio for each type of interval}
\label{table:gaps}
\end{center}
\end{table}

The intervals, gaps and remaining intervals in the example
of Figure \ref{fig:gap}
illustrate the general pattern for all cases in Table \ref{table:gaps}.
The relative order of the remaining intervals and their endpoints
depends on the parity of $n$. The density ratios follow from an easy 
calculation on the sizes of the remaining intervals.

\noindent
\begin{figure}[ht]
\fbox{
\begin{tikzpicture}[style={sibling distance = 3.9cm, level distance = 1.8cm,scale=0.9}]
  \node {\begin{tabular}{c}$I_i$\\ $T_{2,5}([\bara])$\end{tabular}}
    child { node {\begin{tabular}{c}$I_{2i}$\\ $T_{2,3}([\bara])$\end{tabular}}} 
    child { node {\begin{tabular}{c}$I_{2i+1}$\\ $T_{3,5} ([\bara])$\end{tabular}}};
\end{tikzpicture}
}
\fbox{
\begin{tikzpicture}[style={sibling distance = 3.9cm, level distance = 1.8cm,scale=0.9}]
  \node {\begin{tabular}{c}$I_i$\\ $T_{2,5}([\bara])$\end{tabular}}
    child { node {\begin{tabular}{c}$I_{2i}$\\ $T_{3,5} ([\bara])$\end{tabular}}}
    child { node {\begin{tabular}{c}$I_{2i+1}$\\ $T_{2,3}([\bara])$\end{tabular}}} ;
\end{tikzpicture}
}
\caption{Gap for $T_{2, 5}[\bara]$, with $\bara = [a_0;a_1, \ldots, a_n ]$, for $n$ odd and $n$ even}
\label{fig:gap}
\end{figure}
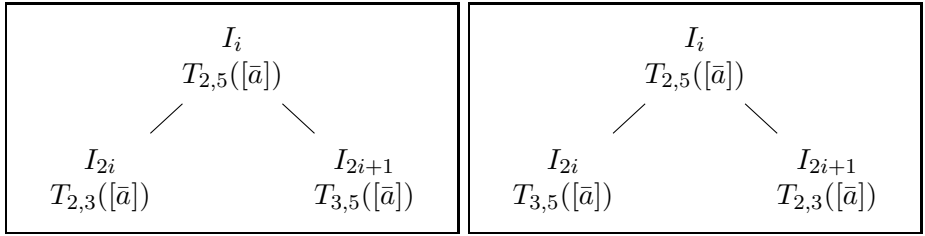

Having described the interval types and the corresponding gaps,
we now create a Cantor set.

\begin{definition}
The Cantor set $\CNF$ will be $(T_{0,1}, \cutw)$.
\end{definition}
\noindent
Note that the size of the initial interval $T_{0,1}$ is at least $0.58256$ by (\ref{eq:initsize})
and the density ratio is at least $1.06122$ (by Table \ref{table:gaps}).

\section{$\NF + \NF = \RR$}\label{sec:equal}
The next goal is to prove our main result, Theorem \ref{maintheorem},
using that $\CNF = \NFS \setminus \QQ$. This will take some preparation,
which culminates in Theorem \ref{thm:sub} and Theorem \ref{thm:sup}.

\begin{lemma} \label{TFinder}
For the intervals $I_i$ in the construction of $\CNF$, for every $n\ge 1$,
the following hold for $[a_0;a_1, \ldots, a_n] \in \NFS$:
\begin{itemize}
\item[\rm (i)] if $a_n \in \{-5,-4,-3,2,3,4\}$, then there exists some 
$i \ge 4^n$ such that \[I_i = T_{a_n,a_n+1}([a_0;a_1, \ldots,a_{n-1}]);\]
\item[\rm (ii)] if $a_n \in \{-4,-3,-2,3,4,5\}$, then there exists some 
$i \ge 4^n$ such that \[I_i = T_{a_n-1,a_n}([a_0;a_1, \ldots,a_{n-1}]).\]
\end{itemize}
\end{lemma}

The proofs of parts (i) and (ii) Lemma \ref{TFinder} are intertwined. 
We prove them using one induction argument.

\begin{proof} 
By induction on $n$.

\vfill\eject
{\bf Base case}: $n = 1$ 

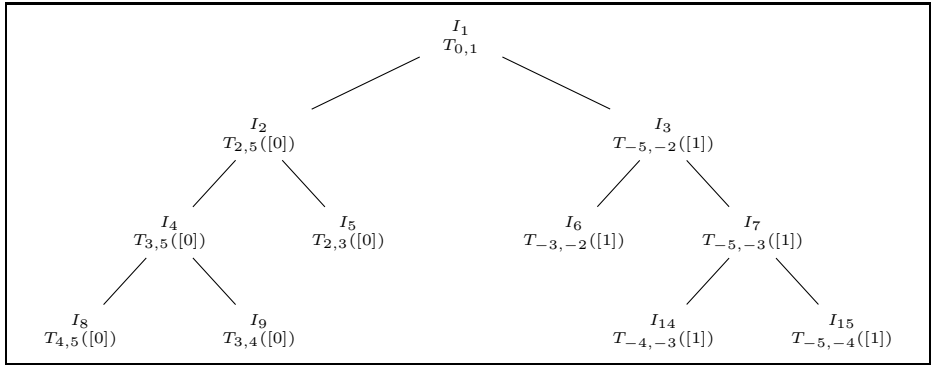
\begin{figure}[ht]\tiny
\fbox{
\begin{tikzpicture}[level/.style={sibling distance = 2cm+1.5cm*(#1-2.5)*(#1-2.5),
  level distance = 1.3cm}]
  \node {\begin{tabular}{c}$I_1$\\ $T_{0,1}$\end{tabular}}
    child { node {\begin{tabular}{c}$I_2$\\ $T_{2,5}([0])$\end{tabular}} 
      child { node {\begin{tabular}{c}$I_4$\\ $T_{3,5}([0])$\end{tabular}} 
        child { node {\begin{tabular}{c}$I_{8}$\\ $T_{4,5}([0])$\end{tabular}} }
        child { node {\begin{tabular}{c}$I_{9}$\\ $T_{3,4}([0])$\end{tabular}} } }
      child { node {\begin{tabular}{c}$I_5$\\ $T_{2,3}([0])$\end{tabular}} } } 
    child { node {\begin{tabular}{c}$I_3$\\ $T_{-5,-2}([1])$\end{tabular}} 
      child { node {\begin{tabular}{c}$I_6$\\ $T_{-3,-2}([1])$\end{tabular}} } 
      child { node {\begin{tabular}{c}$I_7$\\ $T_{-5,-3}([1])$\end{tabular}}
        child { node {\begin{tabular}{c}$I_{14}$\\ $T_{-4,-3}([1])$\end{tabular}} }
        child { node {\begin{tabular}{c}$I_{15}$\\ $T_{-5,-4}([1])$\end{tabular}} } } };
\end{tikzpicture}
}
\end{figure}

If $a_1 \in \{-5,-4,-3,-2\}$ then $a_0 = 1$ and:
$$I_{15} = T_{-5,-4}([1]),\quad I_{14} = T_{-4,-3}([1]),\quad I_6 = T_{-3,-2}([1]).$$

If $a_1 \in \{2,3,4,5\}$ then $a_0 = 0$ and:
$$I_5 = T_{2,3}([0]),\quad I_{9} = T_{3,4}([0]),\quad I_{8} = T_{4,5}([0]).$$

Let  $[a_0;a_1, ..., a_n, a_{n+1}] \in \NFS$, so $a_{n+1} \in \{-5,-4,-3,-2,2,3,4,5\}$.

We make case distinctions based on whether $a_{n+1} \in \{-5,-4,-3,-2\}$ or $a_{n+1} \in \{2,3,4,5\}$, and whether $n$ is odd or even. 

{\bf Case 1}: $a_{n+1} \in \{-5,-4,-3,-2\}$, then $a_n \in \{-4,-3,-2,3,4,5\}$, so by induction there exists $j \ge 4^n$ such that $I_j =T_{a_n-1,a_n}([a_0;a_1, 
\ldots,a_{n-1}])$.

{\bf Case 1a}: $n$ is even

\noindent
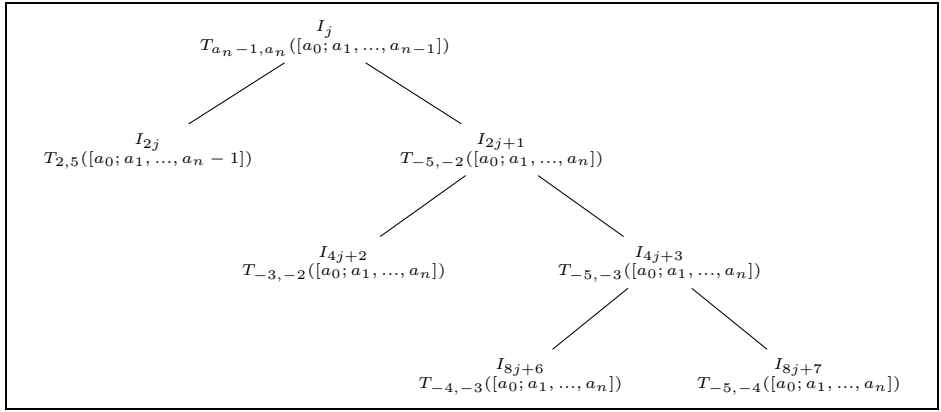
\begin{figure}[h!]\tiny
\fbox{
\begin{tikzpicture}[level/.style={sibling distance = 5.2cm-0.5cm*#1,
  level distance = 1.5cm}]
  \node {\begin{tabular}{c}$I_j$\\ $T_{a_n-1,a_n}([a_0;a_1, ...,a_{n-1}])$\end{tabular}}
    child { node {\begin{tabular}{c}$I_{2j}$\\ $T_{2,5}([a_0;a_1, ...,a_n-1])$\end{tabular}} }
    child { node {\begin{tabular}{c}$I_{2j+1}$\\ $T_{-5,-2}([a_0;a_1, ...,a_n])$\end{tabular}} 
      child { node {\begin{tabular}{c}$I_{4j+2}$\\ $T_{-3,-2}([a_0;a_1, ...,a_n])$\end{tabular}} }
      child { node {\begin{tabular}{c}$I_{4j+3}$\\ $T_{-5,-3}([a_0;a_1, ...,a_n])$\end{tabular}}
        child { node {\begin{tabular}{c}$I_{8j+6}$\\ $T_{-4,-3}([a_0;a_1, ...,a_n])$\end{tabular}} }
        child { node {\begin{tabular}{c}$I_{8j+7}$\\ $T_{-5,-4}([a_0;a_1, ...,a_n])$\end{tabular}} } } };
\end{tikzpicture}
}
\end{figure}

\noindent
We find: $I_{8j+7} = T_{-5,-4}([a_0;a_1, ...,a_n])$,
$I_{8j+6} = T_{-4,-3}([a_0;a_1, ...,a_n])$ and 
$I_{4j+2} = T_{-3,-2}([a_0;a_1, ...,a_n])$.

\newpage
{\bf Case 1b}: $n$ is odd

\noindent
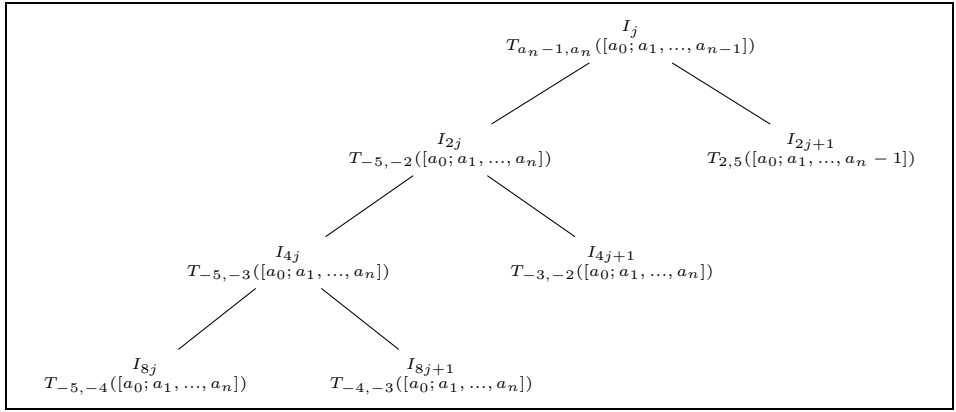
\begin{figure}[h!]\tiny
\fbox{
\begin{tikzpicture}[level/.style={sibling distance = 5.3cm-0.5cm*#1,
  level distance = 1.5cm}]
  \node {\begin{tabular}{c}$I_j$\\ $T_{a_n-1,a_n}([a_0;a_1, ...,a_{n-1}])$\end{tabular}}
    child { node {\begin{tabular}{c}$I_{2j}$\\ $T_{-5,-2}([a_0;a_1, ...,a_n])$\end{tabular}} 
      child { node {\begin{tabular}{c}$I_{4j}$\\ $T_{-5,-3}([a_0;a_1, ...,a_n])$\end{tabular}}
        child { node {\begin{tabular}{c}$I_{8j}$\\ $T_{-5,-4}([a_0;a_1, ...,a_n])$\end{tabular}} }
        child { node {\begin{tabular}{c}$I_{8j+1}$\\ $T_{-4,-3}([a_0;a_1, ...,a_n])$\end{tabular}} } }
      child { node {\begin{tabular}{c}$I_{4j+1}$\\ $T_{-3,-2}([a_0;a_1, ...,a_n])$\end{tabular}} } }
    child { node {\begin{tabular}{c}$I_{2j+1}$\\ $T_{2,5}([a_0;a_1, ...,a_n-1])$\end{tabular}} };
\end{tikzpicture}
}
\end{figure}

\noindent
From this we see
$I_{8j} = T_{-5,-4}([a_0;a_1, ...,a_n])$,
$I_{8j+1} = T_{-4,-3}([a_0;a_1, ...,a_n])$, $I_{4j+1} = T_{-3,-2}([a_0;a_1, ...,a_n])$.

\medskip
{\bf Case 2}: If $a_{n+1} \in \{2,3,4,5\}$, then $a_n \in \{-5,-4,-3,2,3,4\}$, so 
by induction there exists $j \ge 4^n$ such that $I_j =T_{a_n,a_n+1}([a_0;a_1, ...,a_{n-1}])$.

\medskip
{\bf Case 2a}: $n$ is even

\noindent
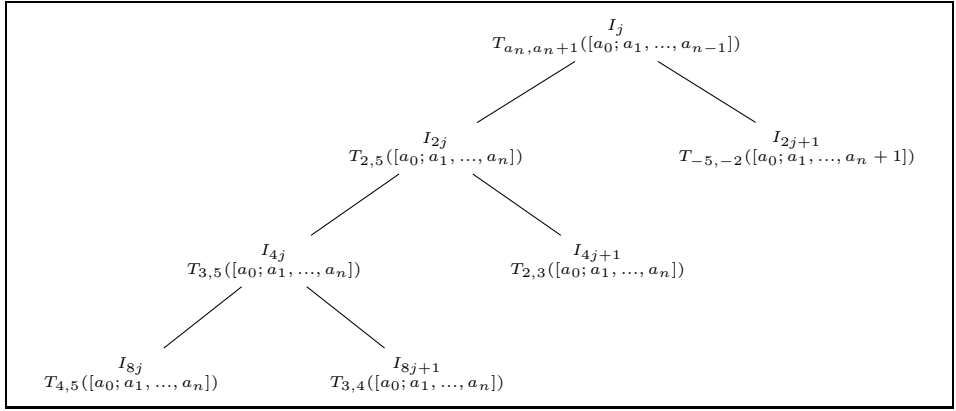
\begin{figure}[h!]\tiny
\fbox{
\begin{tikzpicture}[level/.style={sibling distance = 5.3cm-0.5cm*#1,
  level distance = 1.5cm}]
  \node {\begin{tabular}{c}$I_j$\\ $T_{a_n,a_n+1}([a_0;a_1, ...,a_{n-1}])$\end{tabular}}
    child { node {\begin{tabular}{c}$I_{2j}$\\ $T_{2,5}([a_0;a_1, ...,a_n])$\end{tabular}} 
      child { node {\begin{tabular}{c}$I_{4j}$\\ $T_{3,5}([a_0;a_1, ...,a_n])$\end{tabular}} 
        child { node {\begin{tabular}{c}$I_{8j}$\\ $T_{4,5}([a_0;a_1, ...,a_n])$\end{tabular}} }
        child { node {\begin{tabular}{c}$I_{8j+1}$\\ $T_{3,4}([a_0;a_1, ...,a_n])$\end{tabular}} } }
      child { node {\begin{tabular}{c}$I_{4j+1}$\\ $T_{2,3}([a_0;a_1, ...,a_n])$\end{tabular}} } }
    child { node {\begin{tabular}{c}$I_{2j+1}$\\ $T_{-5,-2}([a_0;a_1, ...,a_n+1])$\end{tabular}} };
\end{tikzpicture}
}
\end{figure}

\noindent
We find in this case $I_{4j+1} = T_{2,3}([a_0;a_1, ...,a_n])$,
$I_{8j} = T_{3,4}([a_0;a_1, ...,a_n])$ and 
$I_{8j+1} = T_{4,5}([a_0;a_1, ...,a_n])$.

\newpage
{\bf Case 2b}: $n$ is odd

\noindent
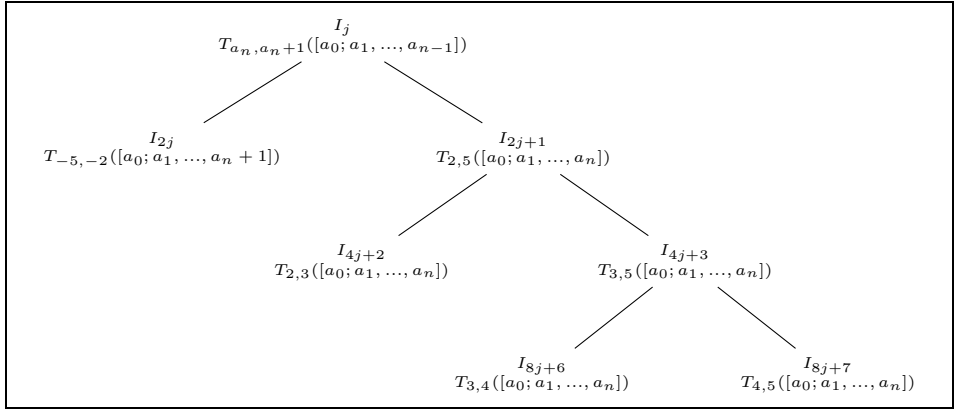
\begin{figure}[h!]\tiny
\fbox{
\begin{tikzpicture}[level/.style={sibling distance = 5.3cm-0.5cm*#1,
  level distance = 1.5cm}]
  \node {\begin{tabular}{c}$I_j$\\ $T_{a_n,a_n+1}([a_0;a_1, ...,a_{n-1}])$\end{tabular}}
    child { node {\begin{tabular}{c}$I_{2j}$\\ $T_{-5,-2}([a_0;a_1, ...,a_n+1])$\end{tabular}} }
    child { node {\begin{tabular}{c}$I_{2j+1}$\\ $T_{2,5}([a_0;a_1, ...,a_n])$\end{tabular}} 
      child { node {\begin{tabular}{c}$I_{4j+2}$\\ $T_{2,3}([a_0;a_1, ...,a_n])$\end{tabular}} }
      child { node {\begin{tabular}{c}$I_{4j+3}$\\ $T_{3,5}([a_0;a_1, ...,a_n])$\end{tabular}} 
        child { node {\begin{tabular}{c}$I_{8j+6}$\\ $T_{3,4}([a_0;a_1, ...,a_n])$\end{tabular}} }
        child { node {\begin{tabular}{c}$I_{8j+7}$\\ $T_{4,5}([a_0;a_1, ...,a_n])$\end{tabular}} } } };
\end{tikzpicture}
}
\end{figure}

\noindent
Here we find: $I_{4j+2} = T_{2,3}([a_0;a_1, ...,a_n])$,
$I_{8j+6} = T_{3,4}([a_0;a_1, ...,a_n])$ and
$I_{8j+7} = T_{4,5}([a_0;a_1, ...,a_n])$.

This concludes our proof.
\end{proof}

\begin{theorem}\label{thm:sub}
  $\NFS \setminus \QQ \subseteq \CNF$
\end{theorem}

\begin{proof} For every $x = [a_0;a_1,...] \in \NFS \setminus \QQ$ we have $[0:\mu] \le x \le [1:-\mu]$, so $x \in T_{0,1}$. 

We continue by showing that there exists no $i$ such that $x \in \cut{I_i}$.

Suppose that there exists an $i$ such that $x \in \cut{I_i}$. Take $n \ge 2$ such that $4^n > i$. We know that $[a_0;a_1,...,a_{n+1}] \in \NFS$.
We make a case distinction whether $a_{n+1} \in \{2,3,4,5\}$ or $a_{n+1} \in \{-5,-4,-3,-2\}$.

Case 1: If $a_{n+1} \in \{2,3,4,5\}$, then $a_n \in \{-5,-4,-3,2,3,4\}$. So by Lemma \ref{TFinder}(i) there exists $j \ge 4^n > i$ with $I_j = T_{a_n,a_n+1}([a_0;a_1,...,a_n-1])$, and $x \in I_j$. Because $j > i$, we have $\cuts{i} \cap I_j = \emptyset$. Contradiction.

Case 2: If $a_{n+1} \in \{-5,-4,-3,-2\}$, then $a_n \in \{-4,-3,-2,3,4,5\}$. By Lemma  \ref{TFinder}(ii) there exists $j \ge 4^n > i$ with $I_j = T_{a_n-1,a_n}([a_0;a_1,...,a_n-1])$, and $x \in I_j$. Because $j > i$, we have $\cuts{i} \cap I_j = \emptyset$. Contradiction.
\end{proof}
As a first step towards proving $\CNF \subseteq \NFS \setminus \QQ$ (Theorem \ref{thm:sup}) we show that for every $x \notin \NFS \setminus \QQ$ we have $x \notin T_{0,1}$ or there exists an $i$ such that $x \in \cut{I_i}$.
For $x \notin \NFS$ we make repeated use of the rules following
Definition \ref{def:nicfstar}.

\begin{proposition} \label{findNICF}
  Given $x = [a_0;a_1, \ldots, a_{n-1}:a_n + r]$, with $0 \le x \le 1$ and $\lvert r \rvert \le \frac{1}{\mu}$. If for each $i \le n$, $\lvert [a_i;a_{i+1}, a_{i+2}, \ldots: a_n + r] \rvert \le \mu$, then:
  $[a_0;a_1, \ldots, a_{n-1},a_n] \in \NFS$ and if $n \ge 1$, then $a_n \in \{-4,-3,-2,2,3,4\}$.
\end{proposition}

\begin{proof}
  We will write $r_i$ (remainder) for $x_i - a_i$ in the construction of $\NCF$, so $r_i = \frac{1}{[a_{i+1};a_{i+2}, \ldots]}$ and $x = [a_0;a_1,\ldots,a_{i-1}:a_i+r_i]$. Note that for every $i$ we have $\lvert r_i \rvert \le \frac{1}{2}$.

  Suppose that $x \notin \NFS$. Each of the following cases then leads to (at least) one contradiction:
  \begin{enumerate}
      \item $a_0 \notin \{0,1\}$: either $a_0 \le -1$, and then $x = a_0 + r_0 \le -\frac{1}{2} < 0$, or $a_0 \ge 2$, and $x = a_0 + r_0 \ge \frac{3}{2} > 1$;
      \item $a_0 = 0$ and $a_{1} < 0$: then $r_0 < 0$, so $x = a_0 + r_0 < 0$;
      \item $a_0 = 1$ and $a_{1} > 0$: then $r_0 > 0$, so $x = a_0 + r_0 > 1$;
      \item $\exists_{1 \le i \le n}$ with $\lvert a_i \rvert \ge 6$: $\lvert [a_i;a_{i+1}, \ldots, a_n] \rvert > \lvert a_i \rvert - r_i \ge 6-\frac{1}{2} > \mu$;
      \item $\exists_{1 \le i < n}$ with $a_i = 5$ and $a_{i+1} > 0$: $r_i > 0$ and $[a_i;a_{i+1}, \ldots, a_n] > 5 > \mu$;
      \item $\exists_{1 \le i < n}$ with $a_i = -5$ and $a_{i+1} < 0$: $r_i < 0$ and $[a_i;a_{i+1}, \ldots, a_n] < -5 < -\mu$.
  \end{enumerate}
  Furthermore, if $\lvert a_n \rvert = 5$, then (take $i = n$) $\lvert [a_n + r] \rvert > 5-\frac{1}{\mu} = \mu$. 
\end{proof}

\begin{proposition} \label{findGap}
  If $x = [a_0;a_1,\ldots,a_n] \in \NFS$ with $n \ge 1$ and $a_n \in \{-4,-3,-2,2,3,4\}$, then 
  for $y = [a_0;a_1,\ldots,a_{n-1}:a_n+r]$ with $\lvert r \rvert < \frac{1}{\mu}$, there exists some $i$ such that $y \in \cut{I_i}$.
\end{proposition}
  
\begin{proof}
  Note that, by construction of $\NCF$, $\lvert a_n+r \rvert \ge 2$.
  \begin{itemize}
      \item If $n = 1$ and $a_1 \in \{2,3,4\}$, then $a_0 = 0$. Since $[0:4+\frac{1}{2}] \le y \le [0;2]$, we know that $y \in I_1$, and that $y \notin I_3$:
      \begin{itemize}
          \item[$\circ$] $a_1 = 2$: $y > [0;2:\mu] = P_{2^+}([0])$, so $y \notin I_2$, thus $y \in \cut{I_1}$;
          \item[$\circ$] $a_1 = 3$: $[0;3:\mu] < y < [0;3:-\mu]$, so $y \in \cut{I_2}$;
          \item[$\circ$] $a_1 = 4$: $[0;4:\mu] < y < [0;4:-\mu]$, so $y \in \cut{I_4}$.
      \end{itemize}
      \item If $n = 1$ and $a_1 \in \{-4,-3,-2\}$, then $a_0 = 1$. Since $[1;-2] \le y \le [1:-4-\frac{1}{2}]$, we know that $y \in I_1$, and that $y \notin I_2$:
      \begin{itemize}
          \item[$\circ$] $a_1 = -2$: $y < [1;-2:-\mu] = P_{-2^-}([1])$, so $y \notin I_3$, thus $y \in \cut{I_1}$;
          \item[$\circ$] $a_1 = -3$: $[1;-3:\mu] < y < [1;-3:-\mu]$, so $y \in \cut{I_3}$;
          \item[$\circ$] $a_1 = -4$: $[1;-4:\mu] < y < [1;-4:-\mu]$, so $y \in \cut{I_7}$.
      \end{itemize}
      \item If $n > 1$, $n$ even and $a_1 \in \{2,3,4\}$, then $a_{n-1} \in \{-5,-4,-3,2,3,4\}$, thus there exists an $i$ such that $I_i = T_{a_{n-1},a_{n-1}+1}([a_0;a_1,\ldots,a_{n-2}])$.
      Since $[a_0;a_1,\ldots,a_{n-1}:4+\frac{1}{2}] < y \le [a_0;a_1,\ldots,a_{n-1},2]$, we know that $y \in I_i$ and that $y \notin I_{2i+1}$:
      \begin{itemize}
          \item[$\circ$] $a_1 = 2$: $[a_0;a_1,\ldots,a_{n-1},2:\mu] < y < [a_0;a_1,\ldots,a_{n-1},2]$, so $y \notin I_{2i}$, thus $y \in \cut{I_i}$;
          \item[$\circ$] $a_1 = 3$: $[a_0;a_1,\ldots,a_{n-1},3:\mu] < y < [a_0;a_1,\ldots,a_{n-1},3:-\mu]$, so $y \in \cut{I_{2i}}$;
          \item[$\circ$] $a_1 = 4$: $[a_0;a_1,\ldots,a_{n-1},4:\mu] < y < [a_0;a_1,\ldots,a_{n-1},4:-\mu]$, so $y \in \cut{I_{4i}}$.
      \end{itemize}
      \item If $n > 1$, $n$ odd and $a_1 \in \{2,3,4\}$, then $a_{n-1} \in \{-5,-4,-3,2,3,4\}$, thus there exists an $i$ such that $I_i = T_{a_{n-1},a_{n-1}+1}([a_0;a_1,\ldots,a_{n-2}])$.
      Since $[a_0;a_1,\ldots,a_{n-1},2] \le y < [a_0;a_1,\ldots,a_{n-1}:4+\frac{1}{2}]$, we know that $y \in I_i$ and that $y \notin I_{2i}$:
      \begin{itemize}
          \item[$\circ$] $a_1 = 2$: $[a_0;a_1,\ldots,a_{n-1},2] \le y < [a_0;a_1,\ldots,a_{n-1},2:\mu]$, so $y \notin I_{2i+1}$, thus $y \in \cut{I_i}$;
          \item[$\circ$] $a_1 = 3$: $[a_0;a_1,\ldots,a_{n-1},3:-\mu] < y < [a_0;a_1,\ldots,a_{n-1},3:\mu]$, so $y \in \cut{I_{2i+1}}$;
          \item[$\circ$] $a_1 = 4$: $[a_0;a_1,\ldots,a_{n-1},4:-\mu] < y < [a_0;a_1,\ldots,a_{n-1},4:\mu]$, so $y \in \cut{I_{4i+3}}$.
      \end{itemize}
      \item If $n > 1$ even and $a_1 \in \{-4,-3,-2\}$, then $a_{n-1} \in \{-4,-3,-2,3,4,5\}$, thus there exists an $i$ such that $I_i = T_{a_{n-1}-1,a_{n-1}}([a_0;a_1,\ldots,a_{n-2}])$.
      Since $[a_0;a_1,\ldots,a_{n-1}:-4-\frac{1}{2}] < y \le [a_0;a_1,\ldots,a_{n-1},-2]$, we know that $y \in I_i$ and that $y \notin I_{2i}$:
      \begin{itemize}
          \item[$\circ$] $a_1 = -2$: $[a_0;a_1,\ldots,a_{n-1},-2:-\mu] < y \le [a_0;a_1,\ldots,a_{n-1},-2]$, so $y \notin I_{2i+1}$, thus $y \in \cut{I_i}$;
          \item[$\circ$] $a_1 = -3$: $[a_0;a_1,\ldots,a_{n-1},-3:\mu] < y < [a_0;a_1,\ldots,a_{n-1},-3:-\mu]$, so $y \in \cut{I_{2i+1}}$;
          \item[$\circ$] $a_1 = -4$: $[a_0;a_1,\ldots,a_{n-1},-4:\mu] < y < [a_0;a_1,\ldots,a_{n-1},-4:-\mu]$, so $y \in \cut{I_{4i+3}}$.
      \end{itemize}
      \item If $n > 1$ odd and $a_1 \in \{-4,-3,-2\}$, then $a_{n-1} \in \{-4,-3,-2,3,4,5\}$, thus there exists an $i$ such that $I_i = T_{a_{n-1}-1,a_{n-1}}([a_0;a_1,\ldots,a_{n-2}])$.
      Since $[a_0;a_1,\ldots,a_{n-1},-2] \le y < [a_0;a_1,\ldots,a_{n-1}:-4-\frac{1}{2}]$, we know that $y \in I_i$ and that $y \notin I_{2i+1}$:
      \begin{itemize}
          \item[$\circ$] $a_1 = -2$: $[a_0;a_1,\ldots,a_{n-1},-2] \le y < [a_0;a_1,\ldots,a_{n-1},-2:-\mu]$, so $y \notin I_{2i}$, thus $y \in \cut{I_i}$;
          \item[$\circ$] $a_1 = -3$: $[a_0;a_1,\ldots,a_{n-1},-3:-\mu] < y < [a_0;a_1,\ldots,a_{n-1},-3:\mu]$, so $y \in \cut{I_{2i}}$;
          \item[$\circ$] $a_1 = -4$: $[a_0;a_1,\ldots,a_{n-1},-4:-\mu] < y < [a_0;a_1,\ldots,a_{n-1},-4:\mu]$, so $y \in \cut{I_{4i}}$.
      \end{itemize}
  \end{itemize}
\end{proof}

\begin{proposition} \label{nmu}
  If $x = [a_0;a_1,\ldots] \in \RR \setminus  (\QQ \cap \NFS)$ then $x < 0$ or $x > 1$ or there exists an $n>0$ such that $\lvert[a_n;a_{n+1},a_{n+2},\ldots]\rvert > \mu$.
\end{proposition}

\begin{proof}
  Take $x = [a_0;a_1,\ldots] \in \RR \setminus  (\QQ \cap \NFS)$, then, because $x \notin \NFS$, (at least) one of the following arguments holds:
  \begin{enumerate}
      \item $a_0 \notin \{0,1\}$: either $a_0 \le -1$, such that $x = a_0 + r_0 \le -\frac{1}{2} < 0$, or $a_0 \ge 2$, such that $x = a_0 + r_0 \ge \frac{3}{2} > 1$;
      \item $a_0 = 0$ and $a_{1} < 0$: $r_0 < 0$, so $x = a_0 + r_0 < 0$;
      \item $a_0 = 1$ and $a_{1} > 0$: $r_0 > 0$, so $x = a_0 + r_0 > 1$;
      \item $\exists_{1 \le i \le n}$ with $\lvert a_i \rvert \ge 6$: $\lvert [a_i;a_{i+1}, \ldots, a_n] \rvert > \lvert a_i \rvert - r_i \ge 6-\frac{1}{2} > \mu$;
      \item $\exists_{1 \le i < n}$ with $a_i = 5$ and $a_{i+1} > 0$: $r_i > 0$ and $[a_i;a_{i+1}, \ldots, a_n] > 5 > \mu$;
      \item $\exists_{1 \le i < n}$ with $a_i = -5$ and $a_{i+1} < 0$: $r_i < 0$ and $[a_i;a_{i+1}, \ldots, a_n] < -5 < -\mu$.
  \end{enumerate}
\end{proof}

\begin{theorem}\label{thm:sup}
  $\CNF \subseteq \NFS \setminus \QQ$
\end{theorem}

\begin{proof}
  We will prove that for every $x \in \RR$, if $x \notin \NFS \setminus \QQ$, then $x \notin I_1 = T_{0,1}$ or there exists an $i$ such that $x \in \cut{I_i}$.
  
  If $x <0$ or $x > 1$ then $x \notin I_1$, so let us assume $0 \le x \le 1$.
  First we are going to show that there exists some $n$ such that $x = [a_0;a_1,\ldots,a_{n-1}:a_n+r]$ and $\lvert r \rvert < \frac{1}{\mu}$:
  \begin{itemize}
    \item Suppose $x \in \QQ$. Then there exists $n$ such that $x = [a_0;a_1,\ldots,a_n] = [a_0;a_1,\ldots:a_n + r]$ where $r= 0$. 
    \item Suppose that $x \notin \QQ$, say $x = [a_0;a_1,\ldots]$, so $x \in \RR \setminus (\QQ \cap \NFS)$. With Theorem \ref{nmu}, there exists $n>0$ such that $\lvert[a_n;a_{n+1},a_{n+2},\ldots]\rvert > \mu$. Let $r = \frac{1}{[a_n;a_{n+1},a_{n+2},\ldots]}$, then $x = [a_0;a_1,\ldots,a_{n-2}:a_{n-1}+r]$ and $\lvert r \rvert < \frac{1}{\mu}$.
  \end{itemize}
  We can define $k$ as the smallest $n$ such that $x = [a_0;a_1,\ldots,a_{n-1}:a_n+r_n]$ with $\lvert r_n \rvert < \frac{1}{\mu}$.
  Now we know that $x = [a_0;a_1,\ldots,a_{k-1}:a_k+r_k]$ and $\lvert r_k \rvert < \frac{1}{\mu}$ and, when we represent $x$ as $[a_0;a_1,\ldots,a_{l-1}:a_l+r_l]$, because $\lvert r_l \rvert \le \frac{1}{\mu}$, we have $[a_{l+1};a_{l+2},\ldots,a_{k-1}:a_k+r_k] > \mu$.

If $k = 0$ then $x = a_0+r$, so $x < 0 + \mu$ or $x > 1 - \mu$, so $x \notin I_1$.
If $k > 0$, then with Theorem \ref{findNICF} we find $[a_0;a_1, \ldots, a_{k-1},a_k] \in \NFS$ and $a_k \in \{-4,-3,-2,2,3,4\}$. Then, from Proposition \ref{findGap}, we know that there exists an $i$ such that $x \in \cut{I_i}$.
\end{proof}

\begin{theorem} \label{nicfsum}
  For every $x \in \mathbb{R} \cap [\frac{2}{\mu},2 - \frac{2}{\mu}]$,  there exist $a,b \in \CNF$ such that $a+b = x$.
\end{theorem}

\begin{proof}
$\CNF$ is a Cantor set with initial interval $[\frac{1}{\mu},1-\frac{1}{\mu}]$, and has a density ratio bigger than 1.

We can now use Theorem \ref{usefulCantor}, as $\frac{x}{x+1} \ge \frac{1}{2}$ if $x \ge 1$, and $\lng{\itvb{1}{1}} = \lng{\itvb{2}{1}}$
\end{proof}

\begin{theorem} 
  For every $x \in \mathbb{R} \cap [\frac{1}{2},\frac{3}{2}]$, there exist $a,b \in \NF$ such that $a + b = x$.
\end{theorem}

\begin{proof}
  Because $[\frac{1}{2},\frac{3}{2}] \subset [\frac{2}{\mu},2 - \frac{2}{\mu}]$, we can apply Theorem \ref{nicfsum}, so there exist $a,b \in \CNF$ such that $a+b = x$. Because $\CNF = \NFS \setminus \QQ$ and $(\NFS \setminus \QQ) \subset \NF$, we have $a,b \in \NF$.
\end{proof}

Now we can prove our main result.

\begin{proof}[Proof of Theorem \ref{maintheorem}]
  We can write $x$ as $y + n$ with $n \in \ZZ$ and $y \in [\frac{1}{2},\frac{3}{2}]$. We know that there exist $a,b \in \NF$ such that $a+b = y$, so $(a+n)+b = x$, with $a+n$ and $b$ in $\NF$.
\end{proof}

\section{$\NFR + \NFR \neq \RR$}\label{sec:not4}

In this chapter, we give a counterexample to $\NFR + \NFR = \RR$.
We will make use of the properties of $\NFR$ similar to
those of $\NF$, following Definition \ref{def:nicfr}.

We denote by $\NFRY$ and $\NFRX$:
\[ \NFRY = \NFR \cap \ [-\frac{1}{2},\frac{1}{2}) \cap \, \QQ\]
and
\[\NFRX = \NFR \cap \ [-\frac{1}{2},\frac{1}{2}) \setminus \QQ,\]
and we will show that:
\begin{enumerate}
    \item $(\NFRX + \NFRX)$ has empty intersection with 
$$[-0.627705, -0.627695] \cup [0.372295 , 0.372305];$$
    \item $\NFRY + \NFRY$ is a countable set;
    \item $\NFRX + \NFRY$ has Lebesgue measure 0.
\end{enumerate}

\begin{lemma}
The smallest positive real number in $\NFRX$ is $[0;\overline{4,2}]$, and the
largest is $[0;\overline{2,4}]$.
\end{lemma}

\begin{proof}
Let $x$ and $y$ the smallest positive and the largest number in $\NFRX$.
We know that the expansions of both will be $[0,a_1, a_2,\ldots]$ with $a_1>0$.
For such elements in $\NFRX$ we know that \[ \frac{1}{a_1 + \frac{1}{2}} \le [0,a_1,a_2,\ldots] \le \frac{1}{a_1 - \frac{1}{2}}.\] Also \[ [0,a_1,a_2,\ldots] = \frac{1}{a_1 + [0;a_2,a_3,\ldots]}, \] so $x = [0:4+y]$ and $y = [0:2+x]$.
The result follows.
\end{proof}

By $\NFR \lbrf a_1,\ldots,a_n \rbrf$ we denote the subset of $\NFRX$ in which the first $n+1$ partial quotients are equal to $0,a_1, \dots, a_n$: 
  \[\NFR \lbrf a_1,\ldots,a_n \rbrf = \{x \in \NFRX |\ \exists y:\ x = [0;a_1,\ldots,a_n:y] \}. \]

The infimum of $\NFR \lbrf a_1,\ldots,a_n \rbrf$ is 
\begin{align*}
	[0;a_1,\ldots,a_n,\overline{4,2}]=[0;a_1,\ldots,a_{n-1}, \overline{2,4}]&\quad \text{ if } a_n = 2;\\
    [0;a_1,\ldots,a_n,\overline{-2,-4}] &\quad \text{ otherwise}.
\end{align*}

The supremum of $\NFR \lbrf a_1,\ldots,a_n \rbrf$ is 
\begin{align*}
	[0;a_1,\ldots,a_n,\overline{-4,-2}]=[0;a_1,\ldots,a_{n-1},\overline{-2,-4}]&\quad \text{ if } a_n = -2;\\
    [0;a_1,\ldots,a_n,\overline{2,4}] &\quad \text{ otherwise}.
\end{align*}

\begin{table}[ht]\tiny
\centering
    \begin{tabular}{rllc}
        fixed coeff.\ \  & \ \ minimum & \ \ maximum & covering interval \\
        $\lbrf-4,-4, -4\rbrf$ & $[0;-4,-4,-4,\overline{2,4}]$ & $[0;-4,-4,-4,\overline{-2,-4}]$ & $[-0.23621, -0.23603] $\\
        $\lbrf-4,-4, -3\rbrf$ & $[0;-4,-4,-3,\overline{2,4}]$ & $[0;-4,-4,-3,\overline{-2,-4}]$ & $[ -0.23654, -0.23623] $\\
        $\lbrf-4,-4, -2\rbrf$ & $[0;-4,-4,\overline{-2,-4}]$ & $[0;-4,-4,-2,\overline{-2,-4}]$ & $[-0.23671, -0.23658] $\\
        $\lbrf-4,-4, 4\rbrf$ & $[0;-4,-4,4,\overline{2,4}]$ & $[0;-4,-4,4,\overline{-2,-4}]$ & $[-0.23448, -0.23425] $\\
        $\lbrf-4,-4, 3\rbrf$ & $[0;-4,-4,3,\overline{2,4}]$ & $[0;-4,-4,3,\overline{-2,-4}]$ & $[-0.23422, -0.23379] $\\
        $\lbrf-4,-4, 2\rbrf$ & $[0;-4,-4,2,\overline{2,4}]$ & $[0;-4,-4,\overline{2,4}]$ & $[-0.23374, -0.23355] $\\
        $\lbrf-4,-3\rbrf$ & $[0;-4,-3,\overline{-2,-4}]$ & $[0;-4,-3,\overline{2,4}]$ & $[-0.23311, -0.22768] $\\
        $\lbrf-4,-2\rbrf$ & $[0;-4,-2,\overline{-2,-4}]$ & $[0;-4,\overline{-2,-4}]$ & $[-0.22685, -0.22474] $\\
        $\lbrf-4, 2\rbrf$ & $[0;-4,\overline{ 2, 4}]$ & $[0;-4,2,\overline{2,4}]$ & $[-0.28165, -0.27841] $\\
        $\lbrf-4, 3\rbrf$ & $[0;-4, 3,\overline{-2,-4}]$ & $[0;-4, 3,\overline{2,4}]$ & $[-0.27717, -0.26953] $\\
        $\lbrf-4, 4, -4\rbrf$ & $[0;-4,4,-4,\overline{2,4}]$ & $[0;-4,4,-4,\overline{-2,-4}]$ & $[-0.26803, -0.26772] $\\
        $\lbrf-4, 4, -3\rbrf$ & $[0;-4,4,-3,\overline{2,4}]$ & $[0;-4,4,-3,\overline{-2,-4}]$ & $[ -0.26862, -0.26806] $\\
        $\lbrf-4, 4, -2\rbrf$ & $[0;-4,4,\overline{-2,-4}]$ & $[0;-4,4,-2,\overline{-2,-4}]$ & $[-0.26894, -0.26870] $\\
        $\lbrf-4, 4, 2\rbrf$ & $[0;-4,4,2,\overline{2,4}]$ & $[0;-4,4,\overline{2,4}]$ & $[-0.26504, -0.26488] $\\
        $\lbrf-4, 4, 3\rbrf$ & $[0;-4,4,3,\overline{2,4}]$ & $[0;-4,4,3,\overline{-2,-4}]$ & $[-0.26548, -0.26508] $\\
        $\lbrf-4, 4, 4\rbrf$ & $[0;-4,4,4,\overline{2,4}]$ & $[0;-4,4,4,\overline{-2,-4}]$ & $[-0.26573, -0.26550] $\\
        $\lbrf-3,-4\rbrf$ & $[0;-3,-4,\overline{-2,-4}]$ & $[0;-3,-4,\overline{2,4}]$ & $[-0.31011, -0.30472] $\\
        $\lbrf-3,-3\rbrf$ & $[0;-3,-3,\overline{-2,-4}]$ & $[0;-3,-3,\overline{2,4}]$ & $[-0.30397, -0.29480] $\\
        $\lbrf-3,-2\rbrf$ & $[0;-3,-2,\overline{-2,-4}]$ & $[0;-3,\overline{-2,-4}]$ & $[-0.29341, -0.28989] $\\
        $\lbrf-3, 2, 2\rbrf$ & $[0;-3,2,2,\overline{2,4}]$ & $[0;-3,2,\overline{2,4}]$ & $[-0.38689, -0.38583] $\\
        $\lbrf-3, 2, 3\rbrf$ & $[0;-3,2,3,\overline{2,4}]$ & $[0;-3,2,3,\overline{-2,-4}]$ & $[-0.39013, -0.38730] $\\
        $\lbrf-3, 2, 4, -4\rbrf$ & $[0;-3,2,4,-4,\overline{-2,-4}]$ & $[0;-3,2,4,-4,\overline{2,4}]$ & $[-0.39086, -0.39073] $\\
        $\lbrf-3, 2, 4, -3\rbrf$ & $[0;-3,2,4,-3,\overline{-2,-4}]$ & $[0;-3,2,4,-3,\overline{2,4}]$ & $[-0.39072, -0.39049] $\\
        $\lbrf-3, 2, 4, -2\rbrf$ & $[0;-3,2,4,-2,\overline{-2,-4}]$ & $[0;-3,2,4,\overline{-2,-4}]$ & $[-0.39046, -0.39036] $\\
        $\lbrf-3, 2, 4, 2\rbrf$ & $[0;-3,\overline{2,4}]$ & $[0;-3,2,4,2,\overline{2,4}]$ & $[-0.39208, -0.39201] $\\
        $\lbrf-3, 2, 4, 3\rbrf$ & $[0;-3,2,4,3,\overline{-2,-4}]$ & $[0;-3,2,4,3,\overline{2,4}]$ & $[-0.39199, -0.39181] $\\
        $\lbrf-3, 2, 4, 4\rbrf$ & $[0;-3,2,4,4,\overline{-2,-4}]$ & $[0;-3,2,4,4,\overline{2,4}]$ & $[-0.39181, -0.39170] $\\
        $\lbrf-3, 3\rbrf$ & $[0;-3,3,\overline{-2,-4}]$ & $[0;-3,3,\overline{2,4}]$ & $[-0.38345, -0.36898] $\\
        $\lbrf-3, 4, -2\rbrf$ & $[0;-3,4,\overline{-2,-4}]$ & $[0;-3,4,-2,\overline{-2,-4}]$ & $[-0.36788, -0.36743] $\\
        $\lbrf-3, 4, -3\rbrf$ & $[0;-3,4,-3,\overline{2,4}]$ & $[0;-3,4,-3,\overline{-2,-4}]$ & $[-0.36727, -0.36623] $\\
        $\lbrf-3, 4, -4\rbrf$ & $[0;-3,4,-4,\overline{2,4}]$ & $[0;-3,4,-4,\overline{-2,-4}]$ & $[-0.36616, -0.36561] $\\
        $\lbrf-3, 4, 4\rbrf$ & $[0;-3,4,4,\overline{2,4}]$ & $[0;-3,4,4,\overline{-2,-4}]$ & $[-0.36189, -0.36147] $\\
        $\lbrf-3, 4, 3\rbrf$ & $[0;-3,4,3,\overline{2,4}]$ & $[0;-3,4,3,\overline{-2,-4}]$ & $[-0.36142, -0.36070] $\\
        $\lbrf-3, 4, 2\rbrf$ & $[0;-3,4,2,\overline{2,4}]$ & $[0;-3,4,\overline{2,4}]$ & $[-0.36061, -0.36032] $\\
        $\lbrf-2,-4\rbrf$ & $[0;\overline{-2,-4}]$ & $[0;-2,-4,\overline{2,4}]$ & $[-0.44949, -0.43827] $\\
        $\lbrf-2,-3\rbrf$ & $[0;-2,-3,\overline{-2,-4}]$ & $[0;-2,-3,\overline{2,4}]$ & $[-0.43671, -0.41804] $\\
        $\lbrf-2,-2\rbrf$ & $[0;-2,-2,\overline{-2,-4}]$ & $[0;-2,\overline{-2,-4}]$ & $[-0.41524, -0.40824] $\\
        $\lbrf2\rbrf$ & $[0;2,\overline{2,4}]$ & $[0;\overline{2,4}]$ & $[0.40824, 0.44949]$\\
        $\lbrf3\rbrf$ & $[0;3,\overline{2,4}]$ & $[0;3,\overline{-2,-4}]$ & $[0.28989, 0.39208] $\\
        $\lbrf4\rbrf$ & $[0;4,\overline{2,4}]$ & $[0;4,\overline{-2,-4}]$ & $[0.22474, 0.28165] $
        \end{tabular}
        \caption{\label{NFRTABLE} Cases of $\NFR\lbrf a_1,\ldots,a_n \rbrf$ including every element of $\NFRX$}
    \end{table}

This makes it possible to determine $\NFR \lbrf a_1,\ldots,a_n \rbrf$
for all possible values of $\lbrf a_1,\ldots,a_n \rbrf$, see Table \ref{NFRTABLE}. Every element of  $\NFRX$ is included in one of these sets, and is
therefore contained in one of the covering intervals listed.

From this it is an easy calculation to find all combinations
of two covering intervals.
None overlaps with either of the intervals $[-0.627705, -0.627695]$ and $[0.372295 , 0.372305]$. We only show the combinations of intervals that are most relevant (see Tables \ref{tab:two} and \ref{tab:three}): for each combination of intervals $I,J$ there exists a sum $I'+J'$ in the list below such that $I' \ge I$ and $J' \ge J$, or $I' \le I$ and $J' \le J$.

\begin{table}[h!]\tiny
\centering
    \begin{tabular}{rcl}
    $\lbrf-4,-2\rbrf + \lbrf2\rbrf =$ &$ [-0.22685, -0.22474] + [0.40824, 0.44949]$&$= [0.18139, 0.22475]$ \\
    $\lbrf4\rbrf + \lbrf4\rbrf =$ &$ [0.22474, 0.28165] + [0.22474, 0.28165] $&$ = [0.44948, 0.56330]$ \\
    \end{tabular}
    \caption{The sums closest to $[0.372295 , 0.372305]$}
\label{tab:three}
\end{table}

\begin{table}[ht]\tiny
\centering
    \begin{tabular}{cll}
    $\lbrf-2,-4\rbrf + \lbrf4\rbrf $&$= [-0.44949, -0.43827] + [0.22474, 0.28165] $&$= [-0.22475, -0.15662] $\\
    $\lbrf-2,-2\rbrf + \lbrf-4,-2\rbrf $&$= [-0.41524, -0.40824] + [-0.22685, -0.22474] $&$= [-0.66356, -0.64278] $\\
    $\lbrf-3, 2, 4, 2\rbrf + \lbrf-4,-4, 4\rbrf $&$= [-0.39208, -0.39201] + [-0.23448, -0.23425] $&$= [-0.62656, -0.62626] $\\
    $\lbrf-3, 2, 4, 4\rbrf + \lbrf-4,-4,-4\rbrf $&$= [-0.39181, -0.39170] + [-0.23621, -0.23603] $&$= [-0.62802, -0.62773] $\\
    $\lbrf-3, 2, 4, -4\rbrf + \lbrf-4,-4,-2\rbrf $&$= [-0.39086, -0.39073] + [-0.23671, -0.23658] $&$= [-0.62757, -0.62731] $\\
    $\lbrf-3, 4, -4\rbrf + \lbrf-4, 4, 2\rbrf $&$= [-0.36616, -0.36561] + [-0.26504, -0.26488] $&$= [-0.63120, -0.63049] $\\
    $\lbrf-3, 4, 4\rbrf + \lbrf-4, 4, 4\rbrf $&$= [-0.36189, -0.36147] + [-0.26573, -0.26550] $&$= [-0.62762, -0.62697] $\\
    $\lbrf-3, 4, 2\rbrf + \lbrf-4, 4, -4\rbrf $&$= [-0.36061, -0.36032] + [-0.26803, -0.26772] $&$= [-0.62864, -0.62804] $\\
    $\lbrf-3, -4\rbrf + \lbrf-3,-4\rbrf $&$= [-0.31011, -0.30472] + [-0.31011, -0.30472] $&$= [-0.62022, -0.60944]$\\
    \end{tabular}
\caption{The sums closest to $[-0.627705, -0.627695]$}
\label{tab:two}
\end{table}

\begin{corollary} \label{intervalsnotreached}
Let $I_1=[-0.627705, -0.627695]$, $I_2=[0.372295 , 0.372305]$, then:
$$(\NFRX+\NFRX) \cap (I_1 \cup I_2) = \emptyset.$$
\end{corollary}

\begin{theorem}
  $\NFR + \NFR \neq \RR$
\end{theorem}

\begin{proof}
It will be clear (from replacing, if necessary, $x, y$ by
$x-\lfloor x \rceil,  y-\lfloor y \rceil$) that $z\in\R$
can be written as $z=x+y$ with $x, y\in \NFR$ if and only if
$z\equiv x+y\bmod 1$ with $x, y\in \NFR \cap \ [-\frac{1}{2},\frac{1}{2})$.

The subset $\NFRY$ of $\QQ$ is countable, and hence
$\NFRY + \NFRY$ has Lebesgue-measure 0.

By a standard argument for Cantor sets, also $\NFRX$ has Lebesgue-measure 0:
it is contained in sequence of finite unions of finite intervals
shrinking to Lebesgue measure zero. As a consequence, also
the set $\{x+y \mid x \in \NFRY, y \in \NFRX \}$ has Lebesgue-measure 0,
by subadditivity of the Lebesgue-measure.
By Corollary \ref{intervalsnotreached} $\NFRX+\NFRX$ misses the
interval $[0.372295 , 0.372305]$ and hence has Lebesgue measure 
at most $0.99999$.

Therefore, $\NFR + \NFR \neq \RR$.
\end{proof}

\end{document}